\def\version{Version 13: 12/12/2007}
\documentclass[11pt,a4paper]{amsart}\usepackage{fullpage}
\usepackage{amssymb,amscd,amsxtra}

\newtheorem{thm}{Theorem}[section]
\newtheorem{lem}[thm]{Lemma}
\newtheorem{prop}[thm]{Proposition}
\newtheorem{cor}[thm]{Corollary}

\theoremstyle{definition}
\newtheorem{rem}[thm]{Remark}

\numberwithin{equation}{section}

\def\ie{\emph{i.e.}}
\def\etal{\emph{et al}}
\def\ds{\displaystyle}
\def\:{\colon}
\def\.{\cdot}
\def\o{\circ}
\def\<{\left\langle}
\def\>{\right\rangle}
\def\({\left(}
\def\){\right)}
\def\ph#1{\phantom{#1}}
\def\epsilon{\varepsilon}
\def\phi{\varphi}
\def\subset{\subseteq}
\def\supset{\supseteq}
\def\leq{\leqslant}
\def\geq{\geqslant}
\def\lra{\longrightarrow}
\def\bar#1{\overline{#1}}
\def\ubar#1{\underline{#1}}
\def\hat#1{\widehat{#1}}
\def\tilde#1{\widetilde{#1}}
\def\iso{\cong}

\DeclareMathOperator{\codom}{codom}
\DeclareMathOperator{\dom}{dom}

\DeclareMathOperator{\rank}{rank}

\def\bbB{{\mathbb B}}
\def\C{{\mathbb C}}
\def\D{{\mathbb D}}
\def\F{{\mathbb F}}
\def\Fpi{{\bar\F_p}}
\def\G{{\mathbb G}}
\def\Gm{{\G_{\mathrm m}}}

\def\k{\Bbbk}
\def\Q{{\mathbb Q}}

\def\W{\mathbb W}
\def\Z{{\mathbb Z}}

\def\Times_#1{\ds\mathop{\times}_{#1}}
\def\oTimes_#1{\ds\mathop{\otimes}_{#1}}
\def\Prod_#1{\ds\mathop{\prod_{#1}}}
\def\ideal{\triangleleft}

\def\wt{\mathrm{wt}}
\def\SL{\mathrm{SL}}
\def\Id{\mathrm{Id}}

\def\rmS{\mathrm{S}}
\DeclareMathOperator{\Aut}{Aut}
\def\AUT{{\mathbf{Aut}}}
\DeclareMathOperator{\End}{End}
\DeclareMathOperator{\Gal}{Gal}
\DeclareMathOperator{\Map}{Map}
\def\Mapc{\Map^{\mathrm{c}}}
\DeclareMathOperator{\Obj}{Obj}
\DeclareMathOperator{\Mor}{Mor}

\DeclareMathOperator{\CW}{CW}
\DeclareMathOperator{\B}{B}

\def\degs{\deg_{\mathrm{s}}}

\def\Alg{\mathbf{Alg}}
\def\Algc{\Alg^{\mathrm{c}}}
\def\CommGpSch{\mathbf{CommGpSch}}
\def\Comod{\mathbf{Comod}}
\def\Comodc{\Comod^{\mathrm{c}}}
\def\DivGp{\mathbf{DivGp}}

\def\ELLFGL{\mathbf{EllFGL}}
\def\SELLFGL{\mathbf{SEllFGL}}
\def\ssELLFGL{{}^{\mathrm{ss}}\ELLFGL}
\def\ssSELLFGL{{}^{\mathrm{ss}}\SELLFGL}
\def\ELL{\mathbf{Ell}}

\def\Mod{\mathbf{Mod}}
\def\Modc{\mathbf{Mod}^{\mathrm{c}}}

\DeclareMathOperator{\Ext}{Ext}
\DeclareMathOperator{\Fr}{Fr}
\DeclareMathOperator{\Hom}{Hom}
\DeclareMathOperator{\IHom}{InvtHom}
\DeclareMathOperator{\Spec}{Spec}
\def\Specf{\Spec^{\mathrm{f}}}
\def\cTensor#1{{\ds\widehat{\mathop{\otimes}_{#1}}}}
\def\Ell{{E\ell\ell}}
\def\ssEll{{{}^{\mathrm{ss}}\Ell}}
\def\ssGamma{{{{}^{\mathrm{ss}}\Gamma}}}
\def\BP{BP}
\def\EO{EO}
\def\MU{MU}
\def\TMF{TMF}
\def\TM{\mathcal T}
\def\VM{\mathcal V}
\def\rmF{\mathrm{F}}
\def\rmV{\mathrm{V}}
\def\invlim#1{\ds\mathop{\varprojlim}_{#1}}
\def\Hasse{\mathrm{Hasse}}
\def\SGp{\mathbb S_2}
\def\tSGp{\widetilde{\mathbb S}_2}
\def\Eta{\mathrm{H}}
\def\modp#1{\bmod{(#1)}}
\def\catC{{\boldsymbol{\mathcal C}}}
\newcommand{\legendre}[2]{\genfrac(){1pt}0{#1}{#2}}
\def\sphat{^{\hat{\ph{B}}}}
\def\Isog{\mathbf{Isog}}
\def\Isogu{\Isog^\times}
\def\tIsog{\tilde\Isog}
\def\tIsogu{\tilde\Isogu}
\def\SepIsog{\mathbf{SepIsog}}
\def\SepIsogu{\SepIsog^\times}
\def\oSepIsog{{\mathbf{SepIsog}}}
\def\oSepIsogu{\oSepIsog^\times}

\def\tSepIsogu{\tilde\SepIsog^\times}

\def\toSepIsogu{\tilde\oSepIsog^\times}
\def\Isogss{\Isog_{\mathrm{ss}}}
\def\Isogssu{\Isogss^\times}

\def\SepIsogss{\SepIsog_{\mathrm{ss}}}
\def\SepIsogssu{\SepIsogss^\times}
\def\oSepIsogss{\oSepIsog_{\mathrm{ss}}}
\def\oSepIsogssu{\oSepIsogss^\times}
\def\tIsogss{\tilde{\Isogss}}

\def\tIsogssu{\tilde{\Isogssu}}

\def\tSepIsogssu{\tilde{\SepIsogssu}}
\def\toSepIsogssu{\tilde{\oSepIsogssu}}
\def\mubf{\boldsymbol{\mu}}

\def\UB#1{\underline{#1}}
\def\Eis{\mathrm{Ei}}
\DeclareMathOperator{\ind}{ind}
\def\Gpd{{\mathcal G}}

\def\GpdN{{\mathcal N}}
\DeclareMathOperator{\Stab}{Stab}
\def\sdp{\rtimes}
\def\Hc{\mathrm{H}_c}
\def\AbsELL{\mathbf{AbsEll}}
\def\d{\mathrm{d}}
\DeclareMathOperator{\Sect}{Sect}

\title[Isogenies of supersingular elliptic curves and elliptic cohomology]
{Isogenies of supersingular elliptic curves over finite fields and operations
in elliptic cohomology}
\author{Andrew Baker}
\address{Department of Mathematics, University of Glasgow,
Glasgow G12~8QW, Scotland.}
\email{a.baker@maths.gla.ac.uk}
\urladdr{http://www.maths.gla.ac.uk/$\sim$ajb}
\thanks{
I would like to thank K.~Buzzard, I.~Connell, J.~Cremona, R.~Odoni, N.~Strickland,
G.~Robert, J.~Tate and V. Gorbounov for their contributions to my understanding of
supersingular elliptic curves over finite fields and also the referee of an earlier
version. \\
{\bf Glasgow University Mathematics Department preprint no. 98/39}
\hfill [\version]}

\begin{document}
\begin{abstract} 
We investigate stable operations in supersingular elliptic cohomology
using isogenies of supersingular elliptic curves over finite fields.
Our main results provide a framework in which we give a conceptually
simple proof of an elliptic cohomology version of the Morava change of
rings theorem and also gives models for explicit stable operations in
terms of isogenies and morphisms in certain enlarged isogeny categories.
We relate our work to that of G.~Robert on the Hecke algebra structure
of the ring of supersingular modular forms. 
\end{abstract}

\maketitle

\section*{Introduction}

In previous work we investigated \emph{supersingular} reductions of elliptic
cohomology~\cite{homell}, stable operations and cooperations in elliptic
cohomology~\cite{phecke,heckop,ellellp1,haell} and in~\cite{ellext,ellcalc}
gave some applications to the Adams spectral sequence based on elliptic
(co)homology. In this paper we investigate stable operations in supersingular
elliptic cohomology using isogenies of supersingular elliptic curves over finite
fields; this is similar in spirit to our earlier work~\cite{ellellp1} on isogenies
of elliptic curves over the complex numbers although our present account is largely
self contained. Indeed, the promised Part~II of~\cite{ellellp1} is essentially
subsumed into the present work together with~\cite{haell,ellext,ellcalc}. A major
inspiration for this work lies in the paper of Robert~\cite{Robert}, which also
led to the related work of~\cite{modforms}; we reformulate Robert's results on the
Hecke algebra structure of the ring of supersingular modular forms in the language
of the present paper.

Throughout, $p$ will be a prime which we will usually assume to be greater than~$3$,
although much of the algebraic theory works as well for the cases $p=2,3$ provided
appropriate adjustments are made. The precise implications for elliptic cohomology
at the primes~$2$ and~$3$ are considerably more delicate and related to work of
Hopkins and Mahowald on the ring of topological modular forms.

A very general discussion of useful background material from algebraic geometry 
can be found in~\cite{Strickland}.

\subsection*{Historical note} 
A version of this paper originally appeared around 1998, and grew out of a long 
period of work on elliptic cohomology in the 1980s sense of a Landweber exact 
cohomology theory based on level~$1$ modular forms with~$6$ inverted. As part 
of this project the author learned a lot about modular forms and their theory and 
tried to apply this to stable homotopy theory. The theory described here was an 
attempt at building a picture of the operations in supersingular elliptic cohomology 
based on work of Tate and others. We did not make use of the modern theory of 
structured ring spectra, so this lacks the spectrum-level rigidity that is now seen 
as crucial in the construction of the topological modular forms spectrum. Nevertheless, 
recent work of Behrens and Lawson~\cite{MB-TL:1,MB-TL:2} has touched on similar ideas 
but in a more sophisticated fashion. We make this paper available on \texttt{arXiv} 
in part to provide a historical record of an earlier attempt at using some of these 
ideas.

\section{Elliptic curves over finite fields}
\label{sec:EllCurFinFlds}

General references for this section are~\cite{Husemoller,Sil},
while~\cite{Katz:ModFunc,Katz-Mazur} provide more abstract formulations. We will
be interested in elliptic curves $\mathcal E$ defined over a subfield $\k\subset\Fpi$,
the algebraic closure of $\F_p$, indeed, we will take $\k=\Fpi$ unless otherwise
specified. We will impose further structure by prescribing a sort of `orientation'
for a curve as part of the data. We will usually assume that $p>3$, although most
of the algebraic details have analogues for the primes~$2$ and~$3$.

We adopt the viewpoint of \cite{Katz:ModFunc,Katz-Mazur}, defining an \emph{oriented
elliptic curve} to be a connected $1$-dimensional abelian group scheme $\mathcal E$
over $\k$ equipped with a nowhere vanishing invariant holomorphic $1$-form
$\omega\in\Omega^1(\mathcal E)$. We will refer to $\mathcal E$ as the underlying
elliptic curve of $\UB{\mathcal E}=(\mathcal E,\omega)$.

A morphism of abelian varieties $\phi\:\mathcal E_1\lra\mathcal E_2$ for which
$\phi^*\omega_2\neq0$ corresponds to a morphism
$\phi\:\UB{\mathcal E_1}\lra\UB{\mathcal E_2}$. Since $\Omega^1(\mathcal E_1)$
is $1$-dimensional over $\Fpi$, there is a unique $\lambda_\phi\in\k^\times$ for
which $\phi^*\omega_2=\lambda_\phi\omega_1$. So such a morphism $\phi$ is a pair
of the form $(\phi,\lambda_\phi)$; if $\lambda_\phi=1$ we will say that $\phi$ is
\emph{strict}. We will denote the category of all such abstract oriented elliptic
curves over $\k$ by $\AbsELL_\k$, and $\AbsELL$ when $\k=\Fpi$.

If $p>3$, by the theory of normal forms to be found in
\cite{Husemoller,Sil,Tate:Arith}, for the oriented elliptic
curve $\UB{\mathcal E}$ there are (non-unique) meromorphic
functions $X,Y$ with poles of orders $2$ and $3$ at
$\mathrm O=[0,1,0]$ and a non-vanishing $1$-form $\d X/Y$
for which
\begin{equation}\label{eqn:Wcub-general}
Y^2=4X^3-aX-b\quad\text{for some $a,b\in\k$}.
\end{equation}
The projectivization $\mathcal E_{\mathrm{W}}$ of the
unique non-singular cubic
\begin{equation}
\label{eqn:WCubic-affine}
y^2=4x^3-ax-b
\end{equation}
is a non-singular Weierstara\ss{} cubic and there is an
isomorphism of elliptic curves
\[
(\theta,\lambda_\theta)\:
(\mathcal E,\d X/Y)\lra(\mathcal E_{\mathrm{W}},\d x/y).
\]
By twisting, we can ensure that $\lambda_\theta=1$, \ie,
$\theta$ is strict. For each $\UB{\mathcal E}$ we choose
such a strict isomorphism $\theta_{\UB{\mathcal E}}$.
Conversely, a Weierstra\ss{} cubic yields an abstract elliptic
curve with the non-vanishing invariant $1$-form $\d x/y$. Let
$\ELL$ denote the full subcategory of $\AbsELL$ consisting of
such Weierstra\ss{} cubics $\mathcal E$ equipped with their
standard $1$-forms, $(\mathcal E,\d x/y)$.
\begin{thm}\label{thm:AbsELL-ELL}
The embedding $\ELL\lra\AbsELL$ is an equivalence of categories.
\end{thm}
Because of this, the phrase (oriented) elliptic curve will now
refer to a Weierstra\ss{} cubic, since we can replace a general
elliptic curve by a Weierstra\ss{} cubic up to isomorphism.


A modular form $f$ of weight $n$ defined over $\k$ is a
rule which assigns to each oriented elliptic curve
$\UB{\mathcal E}=(\mathcal E,\omega)$ over $\k$ a section
$f(\UB{\mathcal E})\omega^{\otimes n}$ of the bundle
$\Omega^1(\mathcal E)^{\otimes n}$, such that for each
isomorphism $\phi\:\mathcal E_1\lra\mathcal E_2$,
\[
\phi^*(f(\UB{\mathcal E_2})\omega_2^{\otimes n})
=f(\UB{\mathcal E_1})\omega_1^{\otimes n}.
\]
In particular, if $\phi^*\omega_2=\lambda\omega_1$, then
\[
f(\UB{\mathcal E_2})=\lambda^{-n}f(\UB{\mathcal E_1}),
\]
which is formally equivalent to $f$ being a modular form
of weight $n$ in the classical sense of \cite{Sil}.

If we rewrite Equation \eqref{eqn:WCubic-affine} in a form
consistent with the notation of~\cite[III \S1]{Sil},
\begin{equation}
\label{eqn:WeierCub}
\mathcal E\:
y^2=
4x^3-\frac{c_4(\UB{\mathcal E})}{12}x+\frac{c_6(\UB{\mathcal E})}{216},
\end{equation}
the functions $c_4,c_6$ are examples of such modular forms of weights~$4$
and~$6$ respectively. The non-vanishing discriminant function $\Delta$
defined by
\[
\Delta(\UB{\mathcal E})=
\frac{(c_4(\UB{\mathcal E})^3-c_6(\UB{\mathcal E})^2)}{1728},
\]
is also a modular form of weight~$12$. In fact the curve $\mathcal E$ is
defined over the finite subfield
$\F_p(c_4(\UB{\mathcal E}),c_6(\UB{\mathcal E}))\subset\Fpi$
and hence over any finite subfield containing it. The \emph{$j$-invariant}
of $\mathcal E$ is
\[
j(\UB{\mathcal E})=
\frac{c_4(\UB{\mathcal E})^3}{\Delta(\UB{\mathcal E})}
\in\F_p(c_4(\UB{\mathcal E}),c_6(\UB{\mathcal E})).
\]
$j$ is a modular form of weight $0$ and only depends on
$\mathcal E$, so we will write $j(\mathcal E)$.

The next result is well known \cite{Husemoller,Sil}; note
that further information is required to determine the
isomorphism class over a finite field containing
$\F_p(c_4(\UB{\mathcal E}),c_6(\UB{\mathcal E}))$.
\begin{thm}
\label{thm:j-invtCompIsoInvt}
The invariant $j(\mathcal E)$ is a complete isomorphism
invariant of the curve $\mathcal E$ over the algebraically
closed field $\Fpi$.
\end{thm}
Another important invariant is the \emph{Hasse invariant}
$\Hasse(\UB{\mathcal E})$ which is a homogeneous polynomial
of weight $p-1$ in $c_4(\UB{\mathcal E}),c_6(\UB{\mathcal E})$
which have given weights 4 and 6 respectively. The oriented
elliptic curve $\UB{\mathcal E}=(\mathcal E,\omega)$ is said
to be \emph{supersingular} if $\Hasse(\UB{\mathcal E})=0$;
again this notion only depends on $\mathcal E$ and not the
1-form $\omega$.

Given $\mathcal E$ defined over $\k\subset\Fpi$, we can consider
$\mathcal E(\k')$, the set of points defined over an extension
field $\k'\supset\k$. We usually regard $\mathcal E(\Fpi)$ as
`the' set of points of $\mathcal E$; thus whenever
$\k\subset\k'\subset\bar\F_p$, we have
\[
\mathcal E(\k)\subset\mathcal E(\k')
\subset\mathcal E(\bar\F_p).
\]
We will also use the notation
\[
\mathcal E[n]=
\ker[n]_{\mathcal E}\:\mathcal E(\Fpi)\lra\mathcal E(\Fpi),
\]
where $[n]_{\mathcal E}\:\mathcal E\lra\mathcal E$ is the multiplication
by $n$ morphism. Actually, this notation is potentially misleading when
$p\mid n$ and should be restricted to the case $p\nmid n$. In
Section \ref{sec:TateMods}, we will also discuss the general case.

For the elliptic curve $\UB{\mathcal E}=(\mathcal E,\omega)$, if meromorphic
functions $X,Y$ are chosen as in Equation \eqref{eqn:Wcub-general}, there
is a local parameter at $\mathrm O$, namely $-2X/Y$, vanishing to order 1
at $\mathrm O$. In terms of the corresponding Weierstra\ss{} form of
Equation \eqref{eqn:WeierCub}, this is the local parameter at
$\mathrm O=[0,1,0]$ given by $t_{\UB{\mathcal E}}=-2x/y$. When referring
to the elliptic curve $\UB{\mathcal E}$, we will often use the notation
\[
(\mathcal E,c_4(\UB{\mathcal E}),c_6(\UB{\mathcal E}),t_{\UB{\mathcal E}})
\]
to indicate that it has Weierstra\ss{} form as in Equation \eqref{eqn:WeierCub}
and local parameter $t_{\UB{\mathcal E}}$. We refer to this data as a
Weierstra\ss{} realization of the elliptic curve
$\UB{\mathcal E}=(\mathcal E,\omega)$.

The local parameter $t_{\UB{\mathcal E}}$ has an associated formal group
law $F_{\UB{\mathcal E}}$ induced from the group structure map
$\mu\:\mathcal E\times\mathcal E\lra\mathcal E$
by taking its local expansion
\[
\mu^*t_{\UB{\mathcal E}}=
F_{\UB{\mathcal E}}(t'_{\UB{\mathcal E}},t''_{\UB{\mathcal E}})
\]
where $t'_{\UB{\mathcal E}},t''_{\UB{\mathcal E}}$ are the local functions
on $\mathcal E\times\mathcal E$ induced from $t_{\UB{\mathcal E}}$ by
projection onto the two factors. Thus we have a formal group law
$F_{\UB{\mathcal E}}(Z',Z'')\in\k[[Z',Z'']]$ if $\mathcal E$ is defined over
$\k$. The coefficients of $F_{\UB{\mathcal E}}$ lie in the $\F_p$-algebra
generated by the coefficients $c_4(\UB{\mathcal E}),c_6(\UB{\mathcal E})$ and
the coefficient of ${Z'}^r{Z''}^s$ is a linear combination of the monomials
$c_4(\UB{\mathcal E})^ic_6(\UB{\mathcal E})^j$ with $4i+6j+1=r+s$ and whose
coefficients are independent of $\UB{\mathcal E}$); in particular, only odd
degree terms in $Z',Z''$ occur.

Given two elliptic curves $\UB{\mathcal E}$ and $\UB{\mathcal E'}$ together
with an isomorphism $\alpha\:\mathcal E\lra\mathcal E'$ of abelian varieties,
there is a new formal group law $F_{\UB{\mathcal E}}^\alpha$ defined by
\[
F_{\UB{\mathcal E}}^\alpha(t_{\UB{\mathcal E}}',
t''_{\UB{\mathcal E}})
=\alpha^{-1}
F_{\UB{\mathcal E'}}(\alpha^*t_{\UB{\mathcal E'}}',
\alpha^*t''_{\UB{\mathcal E'}}).
\]
\begin{lem}
\label{lem:InceFGL}
Let $\UB{\mathcal E}=(\mathcal E,\omega)$ be an oriented elliptic curve and
$\alpha\:\mathcal E\lra\mathcal E$ an automorphism of abelian varieties, then
$F_{\UB{\mathcal E}}^\alpha=F_{\UB{\mathcal E}}$.
\end{lem}
\begin{proof}
The possible absolute automorphism groups are known from~\cite{Husemoller,Sil}
to be given by the following list:
\begin{itemize}
\item
$\Z/6$ if $j(\mathcal E)\equiv0\modp{p}$;
\item
$\Z/4$ if $j(\mathcal E)\equiv1728\modp{p}$;
\item
$\Z/2$ otherwise.
\end{itemize}
In each case, provided that $\F_{p^2}\subset\k$,
$\Aut_{\k}\mathcal E=\Aut\mathcal E$, the absolute automorphism group. The
assertion follows by considering these possibilities in turn; for completeness
we describe them in detail.

When $j(\mathcal E)\equiv0$, the Weierstra\ss{} form is
\[
y^2=x^3+\frac{c_6(\UB{\mathcal E})}{216}
\]
and
\[
F_{\UB{\mathcal E}}(X,Y)=
\sum_{i+j\equiv1\modp{6}}a_{i,j}X^iY^j.
\]
An automorphism of order $6$ has the effect
\[
(x,y)\longmapsto(\zeta_6^2 x,\zeta_6^3 y);
\quad
t_{\UB{\mathcal E}}\longmapsto\zeta_6^{-1}t_{\UB{\mathcal E}},
\]
where $\zeta_6$ is a chosen primitive $6$-th root of unity
in $\Fpi$.

When $j(\mathcal E)\equiv1728$, the Weierstra\ss{} form is
\[
y^2=x^3-\frac{c_4(\UB{\mathcal E})}{12}x,
\]
hence
\[
F_{\UB{\mathcal E}}(X,Y)=
\sum_{i+j\equiv1\modp{4}}a_{i,j}X^iY^j.
\]
An automorphism of order 4 has the effect
\[
(x,y)\longmapsto(\zeta_4^2 x,\zeta_4^3y);
\quad
t_{\UB{\mathcal E}}\longmapsto\zeta_4^{-1}t_{\UB{\mathcal E}},
\]
where $\zeta_4$ is a chosen primitive 4th root of unity in
$\Fpi$.

Finally, in the last case, an automorphism of order~$2$ has
the effect
\[
(x,y)\longmapsto(x,-y);
\quad
t_{\UB{\mathcal E}}\longmapsto-t_{\UB{\mathcal E}}.
\qedhere
\]
\end{proof}

Given a Weierstra\ss{} realization $\mathcal E$ of
$\UB{\mathcal E}$, defined over $\k$, for $u\in\k$,
the curve
\[
\mathcal E^u\:y^2=
4x^3-\frac{u^2c_4(\UB{\mathcal E})}{12}x
+\frac{u^3c_6(\UB{\mathcal E})}{216}
\]
is the \emph{$u$-twist} of $\mathcal E$. For $v\in\k$
with $v^2=u$, there is a \emph{twisting isomorphism}
$\theta_v\:\mathcal E\lra\mathcal E_0$ which is the
completion of the affine map
\[
\phi_v\:(x,y)\longmapsto(v^2x,v^3y).
\]
The effect of this on 1-forms is given by
\[
\theta_v^*\(\frac{dx}{y}\)=v^{-1}\omega.
\]
\begin{thm}
\label{thm:Standard WeierstarssForm}
For each oriented elliptic curve
$\UB{\mathcal E}=(\mathcal E,\omega)$ defined over $\k$,
there is a twisting isomorphism
$\UB{\mathcal E}\lra\UB{\mathcal E_0}$, defined over $\k$
or a quadratic extension $\k'$ of $\k$, where
$\UB{\mathcal E_0}=(\mathcal E_0,dx/y)$ is a Weierstra\ss{}
elliptic curve of one of the following types:
\begin{itemize}
\item
If $j(\mathcal E)\equiv0\modp{p}$,
\[
\mathcal E_0\:y^2=4x^3-4;
\]
\item
if $j(\mathcal E)\equiv1728\modp{p}$,
\[
\mathcal E_0\:y^2=4x^3-4x;
\]
\item
if $j(\mathcal E)\not\equiv0,1728\modp{p}$,
\[
\mathcal E_0\:y^2=
4x^3-\frac{27j(\mathcal E)}{j(\mathcal E)-1728}x
-\frac{27j(\mathcal E)}{j(\mathcal E)-1728}.
\]
\end{itemize}
\end{thm}
\begin{proof}
The above forms are taken from Husemoller \cite{Husemoller}. Given
a Weierstra\ss{} realization $\mathcal E$ of $\UB{\mathcal E}$, it
is easy to see that $\mathcal E$ has the form $\mathcal E_0^u$ for
some $u\in\k$, where $\mathcal E_0$ has one of the stated forms
depending on $j(\mathcal E)$. Then there is a twisting isomorphism
$\theta_v\:\mathcal E\lra\mathcal E_0$ for $v\in\bar\k$ satisfying
$v^2=u$.
\end{proof}

In each of the above cases, the isomorphism
$\theta_v\:\mathcal E\iso\mathcal E_0$ is defined using suitable
choices of twisting parameter $u$. Although this is ambiguous by
elements of the automorphism groups $\Aut\mathcal E\iso\Aut\mathcal E_0$,
we have the following consequence of Lemma~\ref{lem:InceFGL}.
\begin{prop}\label{prop:InceFGL-0}
The formal group law $F_{\UB{\mathcal E}}$
only depends on $\UB{\mathcal E}$, and not on
the isomorphism $\mathcal E\iso\mathcal E_0$, hence
is an invariant of $\UB{\mathcal E}$.
\end{prop}

When $\mathcal E$ is supersingular, we also have the following
useful consequence of the well known fact that
$j(\mathcal E)\in\F_{p^2}$, see~\cite[Chapter V Theorem 3.1]{Sil}.
\begin{prop}\label{prop:FGL-0-GalInce}
If $\mathcal E$ is supersingular, then the coefficients of
$F_{\UB{\mathcal E_0}}$ lie in the subfield
$\F_p(j(\mathcal E))\subset\F_{p^2}$.
\end{prop}

\section{Categories of isogenies over finite fields and their progeny}
\label{sec:CatIsogFinFlds}

For elliptic curves $\UB{\mathcal E_1}$ and $\UB{\mathcal E_2}$
defined over a field $\k$, an \emph{isogeny} (defined over $\k$) is
a non-trivial morphism of abelian varieties
$\phi\:\mathcal E_1\lra\mathcal E_2$.
A \emph{separable isogeny} is an isogeny which is a separable morphism.
This is equivalent to the requirement that $\phi^*\omega_2\neq0$,
where~$\omega_2$ is the non-vanishing invariant $1$-form on $\mathcal E_2$.
An isogeny $\phi$ is finite and the \emph{separable degree} of $\phi$
is defined by
\[
\degs\phi=|\ker\phi|.
\]
If $\phi$ is separable then $\degs\phi=\deg\phi$, the usual notion
of degree.

Associated to the oriented elliptic curve $\UB{\mathcal E}$ over $\Fpi$
defined by Equation~\eqref{eqn:WeierCub}, are the $p^k$-th power curve
\[
\mathcal E^{(p^k)}\:y^2=
4x^3-\frac{1}{12}c_4(\UB{\mathcal E})^{(p^k)}x
+\frac{1}{216}c_6(\UB{\mathcal E})^{(p^k)}
\]
and the $1/p^k$-th power curve
\[
\mathcal E^{(1/p^k)}\:y^2=
4x^3-\frac{1}{12}c_4(\UB{\mathcal E})^{(1/p^k)}x
+\frac{1}{216}c_6(\UB{\mathcal E})^{(1/p^k)}
\]
where for $a\in\Fpi$, $a^{(1/p^k)}\in\Fpi$ is the
unique element satisfying
\[
(a^{(1/p^k)})^{(p^k)}=a.
\]
Properties of these curves can be found in \cite{Sil}.
In particular, given an elliptic curve $\UB{\mathcal E}$,
there is a canonical choice of invariant $1$-forms
$\omega^{(p^k)}$ and $\omega^{(1/p^k)}$ so that the
assignments
\begin{align*}
\UB{\mathcal E}=(\mathcal E,\omega)
&\rightsquigarrow(\mathcal E^{(p^k)},\omega^{(p^k)})
=\UB{\mathcal E}^{(p^k)},\\
\UB{\mathcal E}=(\mathcal E,\omega)
&\rightsquigarrow(\mathcal E^{(1/p^k)},\omega^{(1/p^k)})
=\UB{\mathcal E}^{(1/p^k)}
\end{align*}
extend to functors on the category of isogenies; these
powering operations on $1$-forms can easily be seen in
terms of Weierstra\ss{} forms where they take the canonical
$1$-form $dx/y$ on
\begin{align*}
\mathcal E\:&y^2=
4x^3-\frac{c_4(\UB{\mathcal E})}{12}x
+\frac{c_6(\UB{\mathcal E})}{216}\\
\intertext{to $dX/Y$ on each of the curves}
\mathcal E^{(p^k)}\:&y^2=
4x^3-\frac{1}{12}c_4(\UB{\mathcal E})^{(p^k)}x
+\frac{1}{216}c_6(\UB{\mathcal E})^{(p^k)},\\
\mathcal E^{(1/p^k)}\:&y^2=
4x^3-\frac{1}{12}c_4(\UB{\mathcal E})^{(1/p^k)}x
+\frac{1}{216}c_6(\UB{\mathcal E})^{(1/p^k)}.
\end{align*}
\begin{prop}\label{prop:IsogFactors}
Each isogeny $\phi\:\mathcal E_1\lra\mathcal E_2$ has
unique factorizations
\[
\phi=\Fr^k\o\phi_s={}_s\phi\o\Fr^k
\]
where the morphisms
${}_s\phi\:\mathcal E_1^{(p^k)}\lra\mathcal E_2$,
$\phi_s\:\mathcal E_1\lra\mathcal E_2^{(p^{1/k})}$
are separable and the morphisms denoted $\Fr^k$ are
the evident iterated Frobenius morphisms
$\Fr^k\:\mathcal E_1\lra\mathcal E_1^{(p^k)}$,
$\Fr^k\:\mathcal E_2^{(p^{1/k})}\lra\mathcal E_2$.
\end{prop}

A special case of this is involved in the following.
\begin{prop}\label{prop:Fr2-p}
For a supersingular curve elliptic curve $\mathcal E$
defined over $\k$, the iterated Frobenius
$\Fr^2\:\mathcal E\lra\mathcal E^{(p^2)}$ factors as
\[
\Fr^2\:\mathcal E\xrightarrow{[p]_{\mathcal E}}
\mathcal E\xrightarrow{\lambda}
\mathcal E^{(p^2)},
\]
where $\lambda$ is an isomorphism defined over $\k$. In
particular, if $\mathcal E$ is defined over $\F_{p^2}$
then $\mathcal E^{(p^2)}=\mathcal E$ and $\lambda\in\Aut\mathcal E$.
\end{prop}

Now let $\mathcal E_1$ and $\mathcal E_2$ be defined over
$\Fpi$ and let $\phi\:\mathcal E_1\lra\mathcal E_2$ be a
separable isogeny; then there is a finite field
$\k\subset\Fpi$ such that $\mathcal E_1$, $\mathcal E_2$
and $\phi$ are all defined over $\k$. Later we will make
use of this together with properties of zeta functions of
elliptic curves over finite fields to determine when two
curves over $\Fpi$ are isogenous.

Associated to an isogeny
$\phi\:\mathcal E_1\lra\mathcal E_2$ between two
elliptic curves defined over $\k$ there is a \emph{dual isogeny}
$\hat\phi\:\mathcal E_2\lra\mathcal E_1$ satisfying the
identities
\[
\hat\phi\o\phi=[\deg\phi]_{\mathcal E_1},
\quad
\phi\o\hat\phi=[\deg\phi]_{\mathcal E_2},
\]
where $[n]_{\mathcal E}$ denotes the multiplication by~$n$
morphism on the elliptic curve $\mathcal E$. Localizing the
category of separable isogenies of elliptic curves over finite
fields by forcing every isogeny $[n]_{\mathcal E}$ to be
invertible results in a groupoid since every other regular
isogeny also becomes invertible. Using the theory of $p$-primary
Tate modules, we will modify this construction to define a
larger category which also captures significant $p$-primary
information.

Let $\UB{\mathcal E}$ be an elliptic curve over $\Fpi$ with
a Weierstra\ss{} form as in Equation~\eqref{eqn:WeierCub}
with its associated local coordinate function
$t_{\UB{\mathcal E}}=-2x/y$ and its formal group law
$F_{\UB{\mathcal E}}(X,Y)$. We say that an isogeny
$\phi\:\mathcal E_1\lra\mathcal E_2$ is \emph{strict} if
\[
\phi^*t_{\UB{\mathcal E_2}}\equiv
t_{\UB{\mathcal E_1}}\modp{{t_{\UB{\mathcal E_1}}}^2}.
\]
This condition is equivalent to the requirement that
$\phi^*\omega_2=\omega_1$, hence a strict isogeny is separable.

For a separable isogeny $\phi\:\mathcal E_1\lra\mathcal E_2$
there is a unique factorization of the form
\begin{equation}
\label{eqn:Velu}
\phi\:\mathcal E_1
\xrightarrow{\rho\phantom{'}}\mathcal E_1/\ker\phi
\xrightarrow{\phi'}\mathcal E_2
\end{equation}
where $\phi'$ is an isomorphism, and $\rho$ is a strict isogeny.
The quotient elliptic curve $\mathcal E_1/\ker\phi$ is characterized
by this property and is constructed explicitly by V\'elu~\cite{Velu},
who also determines $\rho^*t_{(\mathcal E_1/\ker\phi,\omega)}$, where
$\omega$ is the $1$-form induced by the quotient map.

We will denote by $\Isog$ the category of elliptic curves over $\Fpi$
with isogenies $\phi\:\mathcal E_1\lra\mathcal E_2$ as its morphisms.
$\Isog$ has the subcategory $\SepIsog$ whose morphisms are the separable
isogenies. These categories have full subcategories $\Isogss$ and
$\SepIsogss$ whose objects are the supersingular curves.

These categories can be localized to produce groupoids. This can be
carried out using dual isogenies and twisting. For the Weierstra\ss{}
cubic $\mathcal E$ defined by Equation~\eqref{eqn:WeierCub}, and a
natural number~$n$ prime to~$p$, the factorization of $[n]_{\mathcal E}$
given by Equation~\eqref{eqn:Velu} has the form
\[
[n]_{\mathcal E}\:\mathcal E
\lra
\mathcal E^{n^2}
\xrightarrow{\tilde{[n]}}\mathcal E
\]
where
\[
\mathcal E^{n^2}\:
y^2=
4x^3-n^{4}\frac{c_4(\UB{\mathcal E})}{12}x
+n^{6}\frac{c_6(\UB{\mathcal E})}{216}
\]
is the twist of $\mathcal E$ by $n^2\in\Fpi^{\times}$
and $\tilde{[n]}$ is the map given by
\[
[x,y,1]\longmapsto[x/n^2,y/n^3,1].
\]
If we invert all such isogenies $[n]_{\mathcal E}$, then
as an isogeny $\phi\:\mathcal E_1\lra\mathcal E_2$ is a
morphism of abelian varieties,
\[
\phi\o[n]_{\mathcal E_1}=[n]_{\mathcal E_2}\o\phi,
\]
hence $\phi$ inherits an inverse
\[
\phi^{-1}=[n]_{\mathcal E_1}^{-1}\o\hat\phi
=\hat\phi\o[n]_{\mathcal E_2}^{-1}.
\]
The resulting localized category of isogenies will be denoted
$\Isogu$ and the evident localized supersingular category
$\Isogssu$. We can also consider the subcategories of separable
morphisms, and localize these by inverting the separable isogenies
$[n]_{\mathcal E}$, \ie, those for which $p\nmid n$. The resulting
categories $\SepIsogu$ and $\SepIsogssu$ are all full subcategories
of $\Isogu$ and $\Isogssu$.

Given $\UB{\mathcal E_1}=(\mathcal E_1,\omega_1)$,
$\UB{\mathcal E_2}=(\mathcal E_2,\omega_2)$, we extend
the action of a separable isogeny
$\phi\:\mathcal E_1\lra\mathcal E_2$ to the morphism
\[
\UB{\phi}=
(\phi,{\phi^*}^{-1})\:\UB{\mathcal E_1}\lra\UB{\mathcal E_2}.
\]
So if $\phi^*\omega_2=\lambda\omega$,
\[
\UB{\phi}(\mathbf x,\omega_1)=
(\phi(\mathbf x),\lambda^{-1}\omega_2).
\]
We will often just write $\phi$ for $\UB{\phi}$ when no
ambiguity is likely to result. Using this construction,
we define modified versions of the above isogeny categories
as follows. $\oSepIsog$ is the category whose objects are the
oriented elliptic curves over $\Fpi$ and with morphisms
\[
(\phi,\lambda^{-1}{\phi^*}^{-1})\:
(\mathcal E_1,\omega_1)\lra(\mathcal E_2,\omega_2),
\]
where $\phi\:\mathcal E_1\lra\mathcal E_2$ is a separable
isogeny and $\lambda\in\Fpi^\times$. Thus $\oSepIsog$ is
generated by morphisms of the form $\UB{\phi}$ together
with the `twisting' morphisms
$\UB{\lambda}\:(\mathcal E,\omega)\lra(\mathcal E,\omega)$
given by $\UB{\lambda}=(\Id_{\mathcal E},\lambda^{-1})$
which commute with all other morphisms. We can localize this
category to form $\oSepIsogu$ with morphisms obtained in an
obvious fashion from those of $\SepIsogu$ together with the
$\UB{\lambda}$. There are also evident full subcategories
$\oSepIsogss$ and $\oSepIsogssu$ whose objects involve only
supersingular elliptic curves.

We end this section with a discussion of two further pieces
of structure possessed by our isogeny categories, both being
actions by automorphisms of these categories. First observe
there is an action of the group of units $\Fpi^\times$ (or more
accurately, the multiplicative group scheme $\Gm$) on $\Isog$
and its subcategories described above, given by
\[
\lambda\.(\mathcal E,\omega)
=(\mathcal E^{\lambda^2},\lambda\omega),
\quad
\lambda\.\phi=\phi^{\lambda^2},
\]
where $\lambda\in\Fpi^\times$,
$\phi\:(\mathcal E_1,\omega_1)\lra(\mathcal E_2,\omega_2)$
is an isogeny and $\phi^{\lambda^2}$ is the evident composite
\[
\phi^{\lambda^2}\:
(\mathcal E_1^{\lambda^{-2}},\lambda^{-1}\omega_1)
\lra(\mathcal E_1,\omega_1)
\xrightarrow{\phi}(\mathcal E_2,\omega_2)
\lra(\mathcal E_2^{\lambda^{2}},\lambda\omega_2).
\]
The second action is induced by the Frobenius morphisms
$\Fr^k$ and their inverses. Namely,
\[
\Fr^k(\mathcal E,\omega)=(\mathcal E^{(p^k)},\omega^{(p^k)}),
\quad
\Fr^k\phi=\phi^{(p^k)},
\]
where for an isogeny
$\phi\:(\mathcal E_1,\omega_1)\lra(\mathcal E_2,\omega_2)$,
$\phi^{(p^k)}$ is the composite
\[
\phi^{(p^k)}\:
(\mathcal E_1^{(1/p^k)},\omega_1^{(1/p^k)})
\xrightarrow{\Fr^{-k}}(\mathcal E_1,\omega_1)
\xrightarrow{\phi}(\mathcal E_2,\omega_2)
\xrightarrow{\Fr^k}(\mathcal E_2^{(p^k)},\omega_2^{(p^k)}).
\]
If $\phi(x,y)=(\phi_1(x,y),\phi_2(x,y))$, then
\[
\phi^{(p^k)}(x,y)=
(\phi_1(x^{1/p^k},y^{1/p^k})^{p^k},
\phi_2(x^{1/p^k},y^{1/p^k})^{p^k}).
\]
Similar considerations apply to the inverse Frobenius morphism
$\Fr^{-k}$.

\section{Recollections on elliptic cohomology}
\label{sec:RecollEllCoh}

A general reference on elliptic cohomology is provided by the
foundational paper of Landweber, Ravenel \& Stong~\cite{LRS},
while aspects of the level~$1$ theory which we use can be found
in Landweber \cite{Land:SS} as well as our earlier
papers~\cite{heckop,homell,ellellp1}.

Let $p>3$ be a prime. We will denote by $\Ell_*$ the graded ring
of modular forms for $\SL_2(\Z)$, meromorphic at infinity and with
$q$-expansion coefficients lying in the ring of $p$-local integers
$\Z_{(p)}$. Here $\Ell_{2n}$ consists of the modular forms of
weight~$n$. We have
\begin{thm}
\label{thm:Ell*Ring}
As a graded ring,
\[
\Ell_*=\Z_{(p)}[Q,R,\Delta^{-1}],
\]
where $Q\in\Ell_8$, $R\in\Ell_{12}$ and
$\Delta=(Q^3-R^2)/1728\in\Ell_{24}$ have the $q$-expansions
\begin{align*}
Q(q)=&E_4=1+240\sum_{1\leq r}\sigma_3(r)q^r,\\
R(q)=&E_6=1-504\sum_{1\leq r}\sigma_5(r)q^r,
\notag\\
\Delta(q)=&q\prod_{n\geq1}(1-q^n)^{24}.
\end{align*}
\end{thm}

The element $A=E_{p-1}\in\Ell_{2(p-1)}$ is particularly important
for our present work. Using the standard notation $B_n$ for the
$n$-th Bernoulli number we have
\[
A(q)=
1-\frac{2(p-1)}{B_{p-1}}\sum_{1\leq r}\sigma_{p-2}(r)q^r
\equiv1\modp{p}.
\]
We also have $B=E_{p+1}\in\Ell_{2(p+1)}$ with $q$-expansion
\[
B(q)=
1-\frac{2(p+1)}{B_{p+1}}\sum_{1\leq r}\sigma_{p}(r)q^r.
\]
Finally, we recall that there is a canonical formal group
law $F_{\Ell}(X,Y)$ defined over $\Ell_*$ whose $p$-series
satisfies
\begin{alignat}3
[p]_{F_{\Ell}}(X)&=pX+\cdots+u_1X^p+\cdots+u_2X^{p^2}
+\text{(higher order terms)}
&&\notag\\
&\equiv u_1X^p+\cdots+u_2X^{p^2}
+\text{(higher order terms)}
&&\modp{p}
\notag\\
&\equiv u_2X^{p^2}
+\text{(higher order terms)}
&&\modp{p,u_1}.
\label{eqn:p-series}
\end{alignat}
Combining results of~\cite{Land:SS,modforms}, we obtain the
following in which $\legendre{-1}{p}$ is the Legendre symbol.
\begin{thm}
\label{thm:pAv2}
The sequence $p,A,B$ is regular in the ring $\Ell_*$,
in which the following congruences are satisfied:
\begin{alignat*}{2}
u_1&\equiv A &&\modp{p};
\\
u_2&\equiv\legendre{-1}{p}\Delta^{(p^2-1)/12}
\equiv-B^{(p-1)}\quad &&\modp{p,A}.
\end{alignat*}
\end{thm}
With the aid of this Theorem together with Landweber's Exact
Functor Theorem, in both its original form \cite{LEFT} and
its generalization due to Yagita \cite{Yagita}, we can define
\emph{elliptic cohomology} and its \emph{supersingular reduction}
by
\begin{align*}
\Ell^*(\ )&=\Ell^*\oTimes_{\BP^*}\BP^*(\ )\\
\ssEll^*(\ )&=(\Ell/(p,A))^*(\ )
\iso\Ell^*/(p,A)\oTimes_{P(2)^*}P(2)^*(\ ),
\end{align*}
where as usual, for any graded group $M_*$ we set
$M^n=M_{-n}$. The structure of the coefficient ring
$\ssEll_*$ was described in \cite{homell} and depends
on the factorization of $A\modp{p}$. In fact, $\ssEll_*$
is a product of `graded fields' and the forms of the
simple factors of $A\modp{p}$ are related to the possible
$j$-invariants of supersingular elliptic curves over $\Fpi$.

Using the definition of supersingular elliptic curves
as pairs $(\mathcal E,\omega)$, an element $f\in\ssEll_{2n}$
can be viewed as a family of sections of bundles
$\Omega^1(\mathcal E)^{\otimes n}$ assigning to
$(\mathcal E,\omega)$ the section
$f(\mathcal E,\omega)\omega^{\otimes n}$. An isomorphism
$\phi\:\mathcal E_1\lra\mathcal E_2$ for which
$\phi^*\omega_2=\lambda\omega_1$ also satisfies
\[
\phi^*f(\mathcal E_2,\omega_2)\omega_2^{\otimes n}
=f(\mathcal E_1,\omega_1)\omega_1^{\otimes n}
\]
and so
\[
\phi^*f(\mathcal E_2,\omega_2)
=\lambda^{-n}f(\mathcal E_1,\omega_1).
\]
This is formally equivalent to $f$ being a modular
form of weight $n$ in the traditional sense.

The ring $\Ell_*/(p)$ is universal for Weierstra\ss{}
elliptic curves defined over $\Fpi$ while $\ssEll_*$
is universal for those which are supersingular, in
the sense of the following result.
\begin{prop}
\label{prop:Ell*UnivProp}
The projectivization $\mathcal E$ of the cubic
\[
y^2=
4x^3-ax-b
\]
defined over $\Fpi$ is an elliptic curve if and only
if there is a ring homomorphisms $\theta\:\Ell_*/(p)\lra\Fpi$
for which
\[
\theta(Q)=12a,\quad\theta(R)=-216b.
\]
For such an elliptic curve, $\mathcal E$ is
supersingular if and only if $\theta(A)=0$.
\end{prop}
The first part amounts to the well known fact that
the discriminant of $\mathcal E$ is
$\Delta(\mathcal E)=a^3-27b^2$, whose non-vanishing
is equivalent to the nonsingularity of $\mathcal E$.
The second part of this result is equivalent to the
statement that $\theta(A)=\Hasse(\UB{\mathcal E})$,
a result which can be found in \cite{Husemoller,Sil}
together with further equivalent conditions.

Next we discuss some cooperation algebras. In \cite{ellellp1},
we gave a description of the cooperation algebra
$\Gamma^0_*=\Ell_*\Ell$ as a ring of functions on the
category of isogenies of elliptic curves defined over
$\C$. We will be interested in the supersingular
cooperation algebra
\[
\ssGamma^0_*=\ssEll_*\Ell=\ssEll_*\oTimes_{\Ell_*}\Ell_*\Ell
\iso\ssEll_*(\Ell).
\]
The ideal $(p,A)\ideal\Ell_*$ is invariant under the
$\Gamma^0_*$-coaction on $\Ell_*$ and hence $\ssGamma^0_*$
can be viewed as the quotient of $\Gamma^0_*$ by the ideal
generated by the image of $(p,A)$ in $\Gamma^0_*$ under
either the left or equivalently the right unit map
$\Ell_*\lra\Gamma^0_*$. The pair $(\ssEll_*,\ssGamma^0_*)$
therefore inherits the structure of a Hopf algebroid over
$\F_p$.

The Hopf algebroid structure on $(\ssEll_*,\ssGamma^0_*)$
implies that
\[
\Spec_{\F_p}\ssGamma^0_*=\Alg_{\F}(\ssGamma^0_*,\Fpi)
\]
is a groupoid, or at least this is so if the grading is ignored.
By the discussion of Devinatz \cite[section 1]{Devinatz} (see
also our Section \ref{sec:CohTopHA}), the grading is equivalent
to an action of $\Gm$ which here is derived from the twisting
action discussed in Section~\ref{sec:EllCurFinFlds}. Let $\ssSELLFGL$
denote the category of supersingular elliptic curves over $\Fpi$
with the morphism set
\[
\ssSELLFGL(\UB{\mathcal E_1},\UB{\mathcal E_2})
=
\{f\:F_{\UB{\mathcal E_1}}\lra F_{\UB{\mathcal E_2}}:
\text{$f$ is a strict isomorphism}\}.
\]
There is an action of $\Gm$ on this extending the twisting action
on curves, and also an action of the Galois group $\Gal(\Fpi/\F_p)$.
These actions are compatible with the composition and inversion maps.
$\ssSELLFGL$ is also a `formal scheme' in the sense used by
Devinatz~\cite{Devinatz}, thus it can be viewed as a pro-scheme and
we can consider continuous functions $\ssSELLFGL\lra\Fpi$ where the
codomain has the discrete topology.
\begin{thm}
\label{thm:GammaReps0}
There is a natural isomorphism of groupoids with $\Gm$-action,
\[
\Spec_{\Fpi}\Fpi\otimes\ssGamma^0_*\iso\ssSELLFGL.
\]
Moreover, $\Fpi\otimes\ssGamma^0_{2n}$ can be identified with the
set of all continuous functions $\ssSELLFGL\lra\Fpi$ of weight~$n$
and $\ssGamma^0_{2n}\subset\Fpi\otimes\ssGamma^0_{2n}$ can be
identified with the subset of Galois invariant functions.
\end{thm}

The proof is straightforward, given the existence of identification
of $\Ell_*\Ell$ as
\[
\Ell_*\Ell=\Ell_*\oTimes_{\MU_*}\MU_*\MU\oTimes_{\MU_*}\Ell_*,
\]
and the universality of $\MU_*\MU$ for strict isomorphisms of formal
group laws due to Quillen \cite{Adams:Chicago,Ravenel:book}. We will
require a modified version of his result.

Recall from~\cite{Adams:Chicago,Ravenel:book} that
\[
\MU_*MU=\MU_*[b_k:k\geq1]
\]
with the convention that $b_0=1$, and that the coaction is determined
by the formula
\[
\sum_{k\geq0}\psi b_kT^{k+1}=
\sum_{r\geq0}1\otimes b_r(\sum_{s\geq0}b_s\otimes1\,T^{s+1})^{r+1}.
\]
This coaction corresponds to composition of power series with leading
term~$T$. We can also form the algebras $\MU_*[u,u^{-1}]$ and
$\MU_*[u,u^{-1}][b_0,b_0^{-1},b_k:k\geq1]$ in which $|u|=|b_0|=0$ and
there is a coaction corresponding to composition of power series with
invertible leading term,
\[
\sum_{k\geq0}\psi b_kT^{k+1}=
\sum_{r\geq0}1\otimes b_r(\sum_{s\geq0}b_s\otimes1\,T^{s+1})^{r+1}.
\]
This also defines a Hopf algebroid
$(\MU_*[u,u^{-1}],\MU_*[u,u^{-1}][b_0,b_0^{-1},b_k:k\geq1])$
whose right unit is given by
\[
\eta_R(xu^n)=\eta_R(x)u^{d+n}b_0^n,
\]
where $x\in\MU_{2d}$ and $\eta_R(x)$ is the image of $x$ under the
usual right unit $\MU_*\lra\MU_*\MU$. There is a ring epimorphism
$\MU_*[u,u^{-1}][b_0,b_0^{-1},b_k:k\geq1]\lra\MU_*\MU$ under which
$u,b_0\longmapsto1$ and which induces a morphism of Hopf algebroids
\[
(\MU_*[u,u^{-1}],\MU_*[b_0,b_0^{-1},b_k:k\geq1])\lra(\MU_*,\MU_*\MU).
\]
We can form a Hopf algebroid $(\Ell_*[u,u^{-1}],\Gamma_*)$ by setting
\[
\Gamma_*=\Ell_*[u,u^{-1}]\oTimes_{\MU_*[u,u^{-1}]}
\MU_*[u,u^{-1}][b_0,b_0^{-1},b_k:k\geq1]
\oTimes_{\MU_*[u,u^{-1}]}\Ell_*[u,u^{-1}]
\]
and there is an induced morphism of Hopf algebroids
\[
(\Ell_*[u,u^{-1}],\Gamma_*)\lra(\Ell_*,\Gamma^0_*).
\]
Similarly, we can define Hopf algebroid $(\ssEll_*[u,u^{-1}],\ssGamma_*)$
with
\[
\ssGamma_*=\ssEll_*[u,u^{-1}]\oTimes_{\MU_*[u,u^{-1}]}
\MU_*[b_0,b_0^{-1},b_k:k\geq1]
\oTimes_{\MU_*[u,u^{-1}]}\ssEll_*[u,u^{-1}].
\]

Now let $\ssELLFGL$ denote the category whose objects are the supersingular
oriented elliptic curves over $\Fpi$ with morphisms being the isomorphisms
of their formal group laws; this category is a topological groupoid with
$\Gm$-action, containing $\ssSELLFGL$. Using the canonical Weierstra\ss{}
realizations of Theorem~\ref{thm:Standard WeierstarssForm}, we have the
following result.
\begin{thm}\label{thm:GammaReps}
There is a natural isomorphism of groupoids with $\Gm$-action,
\[
\Spec_{\Fpi}\Fpi\otimes\ssGamma_*\iso\ssELLFGL.
\]
Moreover, $\Fpi\otimes\ssGamma_{2n}$ can be identified with
the set of all continuous functions $\ssELLFGL\lra\Fpi$ of
weight $n$ and $\ssGamma_{2n}\subset\Fpi\otimes\ssGamma_{2n}$
can be identified with the subset of Galois invariant functions.
\end{thm}

Later we will give a different interpretation of $\ssGamma^0_*$
in terms of the supersingular category of isogenies.

\section{Tate modules}\label{sec:TateMods}

In this section we discuss Tate modules of elliptic curves over
finite fields. While the definition and properties of the Tate
module $\TM_\ell\mathcal E$ for primes $\ell\neq p$ can be found
for example in~\cite{Husemoller,Sil}, we require the details for
$\ell=p$. Suitable references are provided
by~\cite{Waterhouse,WW-JM,Fontaine,Haz}. Actually, it is surprisingly
difficult to locate full details of this material for abelian varieties
in the literature, which seems to have originally appeared in unpublished
papers of Tate \etal.

In this section $\k$ will be a perfect field of characteristic $p>0$
and $\W(\k)$ its ring of \emph{Witt vectors}, endowed with its usual
structure of a local ring (if $\k$ is finite it is actually a complete
discrete valuation ring). The absolute Frobenius automorphism
$x\longmapsto x^p$ on $\k$ lifts uniquely to an automorphism
$\sigma\:\W(\k)\lra\W(\k)$ and we will often use the notation
$x^{(p)}=\sigma(x)$ for this. Let $\D_\k$ be the \emph{Dieudonn\'e
algebra}
\[
\D_\k=\W(\k)\<\rmF,\rmV\>,
\]
\ie, the non-commutative $\W(\k)$-algebra generated by
the elements $\rmF$ and $\rmV$ subject to the relations
\[
\rmF\rmV=\rmV\rmF=p,\quad
\rmF a=a^{(p)}\rmF,\quad
a\rmV=\rmV a^{(p)},
\]
for $a\in\W(\k)$. Let $\Mod_{\D_\k}^{\mathrm{f.l.}}$ be
the category of finite length $\D_\k$-modules and
$\CommGpSch_\k[p]$ be the category of finite commutative
group schemes over $\k$ with rank of the form $p^d$.
\begin{thm}
\label{thm:FGL-TateMod}
There is an anti-equivalence of categories
\begin{align*}
\CommGpSch_\k[p]&\longleftrightarrow\Mod_{\D_\k}^{\mathrm{f.l.}}
\\
G&{\ }\leftrightsquigarrow\mathrm{M}(G).
\end{align*}
Moreover, if $\rank G=p^s$, then $\mathrm{M}(G)$ has
length $s$ as a $\W(\k)$-module.
\end{thm}

This result can be extended to $\DivGp_\k$, the category of
\emph{$p$-divisible groups} over $\k$.
\begin{thm}\label{thm:DivGp-TateMod}
There is an anti-equivalence of categories
\begin{align*}
\DivGp_\k&\longleftrightarrow\Mod_{\D_\k}^{\mathrm{f.l.}}\\
G&{\ }\leftrightsquigarrow\mathrm{M}(G).
\end{align*}
Moreover, if $\rank G=p^s$, $\mathrm{M}(G)$ is a free
$\W(\k)$-module of rank $s$.
\end{thm}

A $p$-divisible group $G$ of rank $p^s$ is a collection of
finite group schemes $G_n$ ($n\geq0$) with $\rank G_n=p^{ns}$
and exact sequences of abelian group schemes
\[
0\lra G_n\xrightarrow{j_n}G_{n+1}\lra G_1\lra0
\]
for $n\geq0$. The extension of the result to such groups is
accomplished by setting
\[
\mathrm{M}(G)=\invlim{n}\mathrm{M}(G_n)
\]
where the limit is taken over the inverse system of maps
$\mathrm{M}(j_n)\:\mathrm{M}(G_{n+1})\lra\mathrm{M}(G_n)$.
The main types of examples we will be concerned with here
are the following.

If $F$ is a $1$-dimensional formal group law over $\k$ of
height $h$, then the $p^n$-series of $F$ has the form
\begin{equation}
\label{eqn:p-n-series}
[p^n]_F(X)\equiv uX^{p^{nh}}\modp{X^{p^{nh}+1}}
\end{equation}
where $u\in\k^\times$. We have an associated $p$-divisible
group $\ker[p^\infty]_F$ of rank $p^h$ with
\[
(\ker[p^\infty]_F)_n=\ker[p^n]_F=\Spec(\k[[X]]/([p^n]_F(X))).
\]

Let $\mathcal E$ be a supersingular elliptic curve defined
over $\k$. The family of finite group schemes
\[
\mathcal E[p^n]=\ker[p^n]_{\mathcal E}\quad(n\geq0)
\]
constitute a $p$-divisible group $\mathcal E[p^\infty]$ of
rank $p^2$. In particular, if $F_{\mathcal E}$ is the formal
group law associated to the local parameter $t_{\mathcal E}$
associated with a Weierstra\ss{} equation for $\mathcal E$,
we have
\begin{lem}\label{lem:EllCur-DivGps}
If $\mathcal E$ is a supersingular elliptic curve defined over
$\k$, there is a compatible family of isomorphisms of group
schemes over $\k$,
\[
\mathcal E[p^n]\iso\ker[p^n]_{F_{\mathcal E}}
\quad(n\geq0),
\]
and hence there is an isomorphism of divisible groups
\[
\mathcal E[p^\infty]\iso\ker[p^\infty]_{F_{\mathcal E}}.
\]
\end{lem}
\begin{proof}
This is essentially proved in
Silverman~\cite[Chapter VII Proposition 2.2]{Sil}; see also
Katz \& Mazur~\cite[Theorem 2.3.2]{Katz-Mazur}. If $\mathcal E$
is given by a Weierstra\ss{} Equation \eqref{eqn:WeierCub},
then in terms of the local parameter $t=-x/y$, $y$ can be
expressed in the form $y=-1/w(t)$ for some power series
$w(t)\in\k[[t]]$ satisfying
\[
w(t)\equiv t^3\modp{t^4}.
\]
By Equation~\eqref{eqn:p-n-series}, the assignment
\[
t\longmapsto\left(\frac{t}{w(t)},\frac{-2}{w(t)}\right)
\]
extends to a $\k$-algebra homomorphism
\[
\k[[t]]/([p^n]_{F_{\mathcal E}}(t))
\lra
\mathcal E(\k[[t]]/(t^{p^{nh}})
\]
where $h$ is the height of $F_{\mathcal E}$, known to be
$1$ when $\mathcal E$ is an ordinary curve or $2$ when it
is supersingular. This induces a homomorphism of $\k$-schemes
\[
\mathcal E(\k[[t]]/(t^{p^{nh}})[p^n]
\iso
\ker[p^n]_{F_{\mathcal E}}
\]
and Silverman's argument applied to the complete local ring
$R=\k[[t]]/(t^{p^{nh}})$ shows this to be an isomorphism.

An alternative approach to proving this makes use of the Serre-Tate
theory described in Katz~\cite[Theorem 1.2.1]{Katz:Serre-Tate},
together with Silverman~\cite[Chapter~VII Proposition 2.2]{Sil}.
\end{proof}

We can now define the \emph{Tate module} of the supersingular
elliptic curve $\mathcal E$ to be
\[
\TM_p\mathcal E=\mathrm M(\mathcal E[p^\infty])
\iso
\mathrm M(\ker[p^\infty]_{F_{\mathcal E}}).
\]
\begin{prop}\label{prop:TM-EllCur}
The Tate module $\TM_p\mathcal E$ is a free topological
$\W(\k)$-module of rank $2$.
\end{prop}

In its strongest form, the following result from~\cite{Waterhouse}
is due to J.~Tate, although never formally published by him; weaker
variants were established earlier by Weil and others; a proof
appears in~\cite{WW-JM}.
\begin{thm}\label{thm:Tate}
Let $\mathcal E$ and $\mathcal E'$ be elliptic curves over $\F_{p^d}$.
Then the natural map
\[
\Hom_{\F_{p^d}}(\mathcal E,\mathcal E')
\lra
\Hom_{\D_{\F_{p^d}}}(\TM_p\mathcal E',\TM_p\mathcal E)
\]
is injective and the induced map
\[
\Hom_{\F_{p^d}}(\mathcal E,\mathcal E')\otimes\Z_p
\lra
\Hom_{\D_{\F_{p^d}}}(\TM_p\mathcal E',\TM_p\mathcal E)
\]
is an isomorphism.
\end{thm}

Since $\Hom_{\F_{p^d}}(\mathcal E,\mathcal E')$ is a free
abelian group of finite rank,
$\Hom_{\F_{p^d}}(\mathcal E,\mathcal E')\otimes\Z_p$ agrees
with its $p$-adic completion
$\Hom_{\F_{p^d}}(\mathcal E,\mathcal E')\sphat_p$. This
finiteness also implies that for sufficiently large~$d$,
\[
\Hom_{\F_{p^d}}(\mathcal E,\mathcal E')
=\Hom_{\Fpi}(\mathcal E,\mathcal E').
\]
So interpreting $\F_{p^\infty}$ as $\Fpi$, Theorem~\ref{thm:Tate}
also holds in that case.

The above definition of $\TM_p\mathcal E$ is different in essence
from that of the Tate modules
\[
\TM_\ell\mathcal E=\invlim{n}\mathcal E[\ell^n]
\]
for primes $\ell\neq p$. However, for any $\k$-algebra $S$, we
may follow Fontaine~\cite[Chapitre V]{Fontaine} and consider
\[
\TM'_p\mathcal E(S)=
\Hom_{\Z_p}(\Q_p/\Z_p,\mathcal E[p^\infty](S)).
\]
In his notation and terminology, Fontaine shows that the functor
$\TM'_p\mathcal E$ satisfies
\[
\TM'_p\mathcal E(S)=
\Hom^{\mathrm{cont}}_{\D_\k}
(\Q_p/\Z_p\oTimes_{\Z_p}\TM_p\mathcal E[p^\infty](S),
\CW_\k(S)).
\]
{}From this it can be deduced that the case $\ell=p$ of
the following result holds, the case where $\ell\neq p$
being dealt with in~\cite{Sil,Husemoller}.
\begin{thm}\label{thm:Tate2}
Let $\mathcal E$ and $\mathcal E'$ be elliptic curves over
$\F_{p^d}$ and for a prime $\ell$ let
\[
\TM'_\ell\mathcal E=
\begin{cases}
\invlim{n}\mathcal E[\ell^n]&\text{if $\ell\neq p$}, \\
\TM'_p\mathcal E&\text{if $\ell=p$}.
\end{cases}
\]
Then the natural map
\[
\Hom_{\F_{p^d}}(\mathcal E,\mathcal E')
\lra
\Hom_{\Gal(\Fpi/\F_{p^d})}(\TM'_\ell\mathcal E,\TM'_\ell\mathcal E')
\]
is injective and the induced map
\[
\Hom_{\F_{p^d}}(\mathcal E,\mathcal E')\otimes\Z_\ell
\lra
\Hom_{\Gal(\Fpi/\F_{p^d})}(\TM'_\ell\mathcal E,\TM'_\ell\mathcal E')
\]
is an isomorphism.
\end{thm}

If $\mathcal E$ is a supersingular elliptic curve defined over $\Fpi$,
its absolute endomorphism ring $\End\mathcal E=\End_{\Fpi}\mathcal E$
is a maximal order in a quaternion division algebra over $\Q$. On passing
to the $p$-adic completion of $\End\mathcal E$, we obtain a non-commutative
$\W(\F_{p^2})$-algebra of rank~$2$,
\[
\mathcal O_{\mathcal E}=\End\mathcal E\otimes\Z_p.
\]
\begin{prop}\label{prop:EndEpCompletion}
The division algebra $\End\mathcal E\otimes\Q$ is unramified except at~$p$
and $\infty$.

If $\mathcal E$ is defined over $\F_{p^d}$, then as a $\W(\F_{p^2})$-algebra,
the $p$-adic completion $\mathcal O_{\mathcal E}$ is given by
\[
\mathcal O_{\mathcal E}=\W(\F_{p^2})\<\Fr^{(d)}\>,
\]
where $\Fr^{(d)}$ is the relative Frobenius map
$\Fr^{(d)}\:\mathcal E\lra\mathcal E^{(p^d)}$ which satisfies the
relations
\[
\begin{cases}
{\Fr^{(d)}}^2=up^d
&\text{with $u$ a unit in $\W(\F_{p^2})$},\\
\Fr^{(d)}\alpha=\alpha^{(p^d)}\Fr^{(d)}&
\text{for all $\alpha\in\W(\F_{p^2})$}.
\end{cases}
\]
When $d=1$, $\mathcal O_{\mathcal E}=\W(\F_{p^2})\<\rmS\>$ is also
isomorphic to the $\W(\F_{p^2})$-algebra $\D_{\F_{p^2}}$ with $\rmS$
corresponding to the Frobenius element $\rmF$ and agreeing with $\Fr$
up to a unit in $\W(\F_{p^2})$.
\end{prop}
\begin{proof}
See~\cite[Chapters 2 \& 4]{Waterhouse}.
\end{proof}
Notice that $\mathcal O_{\mathcal E}$ has a natural $p$-adic topology
extending that of $\Z_p$. Moreover, every element
$\alpha\in\mathcal O_{\mathcal E}$ has a unique Teichm\"uller expansion
\begin{equation}\label{eqn:Teichmueller}
\alpha=\alpha_0+\alpha_1S
\quad
(\alpha_0\in\W(\F_{p^2}),\,\alpha_0^{p^2}=\alpha_0).
\end{equation}

As consequence of Proposition~\ref{prop:EndEpCompletion}, the formal
group law $F_{\UB{\mathcal E}}$ becomes a \emph{formal $\W(\F_{p^2})$-module}
as defined in Hazewinkel~\cite{Haz}. We set
\[
\End F_{\UB{\mathcal E}}=\End_{\Fpi}F_{\UB{\mathcal E}}.
\]
\begin{prop}\label{prop:EllFGL-WFq-module}
The natural homomorphism $\End\mathcal E\lra\End F_{\UB{\mathcal E}}$
extends to an isomorphism of $\W(\F_{p^2})$-algebras
$\mathcal O_{\mathcal E}\lra\End F_{\UB{\mathcal E}}$.
\end{prop}
\begin{proof}
The extension to a map on the $p$-adic completion is straightforward,
and the fact that the resulting map is an isomorphism uses Lemma~\ref{lem:EllCur-DivGps}
together with Tate's Theorem~\ref{thm:Tate}; see also Katz~\cite[\S IV]{Katz:Cartier}.
\end{proof}
\begin{cor}\label{cor:TateModuleFormalWFq}
$\TM_p\mathcal E$ is a module over the $\Z_p$-algebra
$(\W(\Fpi)\oTimes_{\Z_p}\W(\F_{p^2}))\<\rmS\>$ in which
\[
\rmS(\alpha\otimes\beta)
=\alpha^{(p)}\otimes\beta^{(p)}\rmS.
\]
\end{cor}
\begin{proof}
Elements of $\W(\F_{p^2})\subset\End\mathcal E$ induce morphisms of
$\TM_p\mathcal E$. By definition of the Frobenius operation $\rmF$
in~\cite[Chapter III \S5]{Demazure}, we obtain the stated intertwining
formula.
\end{proof}

Using Corollary~\ref{cor:TateModuleFormalWFq}, we can deduce more on
the structure of $\TM_p\mathcal E$. Let $\Gamma=\Gal(\Fpi/\F_p)\iso\hat\Z$
and $\Eta=\Gal(\Fpi/\F_{p^2})$, hence $\Gamma/\Eta\iso\Z/2$.
By~\cite[Chapitre III Proposition 2.1]{Fontaine} extended in the obvious
way to the infinite dimensional situation, the multiplication map
\[
\W(\Fpi)\oTimes_{\W(\F_{p^2})}(\TM_p\mathcal E)^\Eta
\lra\TM_p\mathcal E
\]
is an isomorphism of $\W(\Fpi)$-modules. Indeed it is an isomorphism
of (topological) left $\Gamma$-modules and indeed of
$\W(\Fpi)\oTimes_{\Z_p}\W(\F_{p^2})$-modules. Moreover, viewed as a
module over the right hand factor of
\[
\W(\F_{p^2})\iso1\oTimes_{\Z_p}\W(\F_{p^2})
\subset\W(\Fpi)\oTimes_{\Z_p}\W(\F_{p^2})
\]
it is free of rank $2$. We can now deduce from this that for any pair
of supersingular elliptic curves $\mathcal E,\mathcal E'$,
$\Hom_{\D_{\F_{p^d}}}(\TM_p\mathcal E',\TM_p\mathcal E)$ is a free
module of rank~$1$ over $\W(\F_{p^2})\<\rmS\>$, hence it is a free module
of rank~$2$ over $\W(\F_{p^2})$.

The ring $\W(\F_{p^2})\<\rmS\>$ is familiar to topologists as the absolute
endomorphism ring of the universal Lubin-Tate formal group law height~$2$,
agreeing with that of the natural orientation in Morava $K(2)$-theory. Its
group of units is
\[
\SGp=
\{\alpha_0+\alpha_1\rmS\in\W(\F_{p^2})\<\rmS\>:
\alpha_0,\alpha_1\in\W(\F_{p^2}),\;\alpha_0\not\equiv0\modp{p}\},
\]
while
\[
\SGp^0=
\{\alpha_0+\alpha_1\rmS\in\W(\F_{p^2})\<\rmS\>:
\alpha_0,\alpha_1\in\W(\F_{p^2}),\;\alpha_0\equiv1\modp{p}\}
\]
is its group of \emph{strict units}, known to topologists as the \emph{Morava
stabilizer group}. Let $\bbB_2$ be the rationalization of $\W(\F_{p^2})\<\rmS\>$
which is a $4$-dimensional central division algebra over $\Q_p$. Adopting
notation of~\cite{haell}, we will also introduce the following closed subgroup
of the group of units $\tSGp=\bbB_2^\times$:
\[
\tSGp^0=\bigcup_{r\in\Z}\SGp^0\,\rmS^r.
\]
Notice also that
\[
\tSGp=
\bigcup_{r\in\Z}\SGp\,\rmS^r.
\]
Then $\SGp\ideal\tSGp$ and $\SGp^0\ideal\tSGp^0$, \ie, these are
closed normal subgroups.

We can rationalize the Tate module $\TM_p\mathcal E$, to give
\[
\VM_p\mathcal E=\Q_p\oTimes_{\Z_p}\TM_p\mathcal E,
\]
which is a $2$-dimensional vector space over the fraction field
$\B(\k)=\Q_p\oTimes_{\Z_p}\W(\k)$. In fact, $\VM_p\mathcal E$ is
a module over the rationalization $\bbB_\k=\Q_p\oTimes_{\Z_p}\D_\k$.
We can generalize Tate's Theorem to give the following, which we
only state for curves defined over $\Fpi$.
\begin{thm}\label{thm:TateGeneral}
Let $\UB{\mathcal E_1}$ and $\UB{\mathcal E_2}$ be elliptic curves
over $\Fpi$. Then the natural map
\[
\Isogu(\UB{\mathcal E_1},\UB{\mathcal E_2})
\lra
\Hom_{\bbB_{\Fpi}}(\VM_p\mathcal E_2,\VM_p\mathcal E_1)
\]
is injective and has image contained in
$\IHom_{\bbB_{\Fpi}}(\VM_p\mathcal E_2,\VM_p\mathcal E_1)$, the
open subset of invertible homomorphisms, and the induced map
\[
\Isogu(\UB{\mathcal E_1},\UB{\mathcal E_2})\sphat_p
\lra
\IHom_{\bbB_{\Fpi}}(\VM_p\mathcal E_2,\VM_p\mathcal E_1)
\]
is a homeomorphism.

The natural map
\[
\SepIsogu(\UB{\mathcal E_1},\UB{\mathcal E_2})
\lra
\Hom_{\D_{\Fpi}}(\TM_p\mathcal E_2,\TM_p\mathcal E_1)
\]
is injective with image contained in the open set of
invertible homomorphisms and the induced map
\[
\SepIsogu(\UB{\mathcal E_1},\UB{\mathcal E_2})\sphat_p
\lra
\IHom_{\D_{\Fpi}}(\TM_p\mathcal E_2,\TM_p\mathcal E_1)
\]
is a homeomorphism.

Similar results hold for the supersingular isogeny categories.
\end{thm}

Every separable isogeny induces an isomorphism of Tate modules,
therefore the image of $\SepIsog(\mathcal E_1,\mathcal E_2)$ in
$\Hom_{\D_{\Fpi}}(\TM_p\mathcal E_2,\TM_p\mathcal E_1)$ is a dense
subset of $\IHom_{\D_{\Fpi}}(\TM_p\mathcal E_2,\TM_p\mathcal E_1)$.

Notice also that
$\alpha\in\IHom_{\D_{\Fpi}}(\TM_p\mathcal E,\TM_p\mathcal E)$ with
$\alpha=\alpha_0+\alpha_1S$ as in Equation~\eqref{eqn:Teichmueller},
has the well defined effect $\alpha^*\omega=\alpha_0\omega$ on $1$-forms,
since this is certainly true for elements of the dense subgroup
$\End\mathcal E^\times\subset\IHom_{\D_{\Fpi}}(\TM_p\mathcal E,\TM_p\mathcal E)$.

\section{Thickening the isogeny categories}\label{sec:Thickenings}

Theorem~\ref{thm:TateGeneral} leads us to define the following `thickenings'
of our isogeny categories. Starting with the category $\Isog$ and its
subcategories, we enlarge each to a subcategory of $\tIsog$, with the
same objects but as the set of morphisms $\mathcal E_1\lra\mathcal E_2$,
\begin{align*}
\tIsog(\mathcal E_1,\mathcal E_2)
&=
\Hom_{\D_{\Fpi}}(\TM_p\mathcal E_2,\TM_p\mathcal E_1)-\{0\},\\
\tIsogu(\mathcal E_1,\mathcal E_2)&=
\IHom_{\bbB_{\Fpi}}(\VM_p\mathcal E_2,\VM_p\mathcal E_1),\\
\tSepIsogu(\mathcal E_1,\mathcal E_2)&=
\IHom_{\D_{\Fpi}}(\TM_p\mathcal E_2,\TM_p\mathcal E_1),
\end{align*}
and similarly for the supersingular categories.

We can make similar constructions for the categories whose objects
are oriented elliptic curves, setting
\begin{align*}
\toSepIsogu((\mathcal E_1,\omega_1),(\mathcal E_2,\omega_2))
&=
\IHom_{\D_{\Fpi}}(\TM_p\mathcal E_2,\TM_p\mathcal E_1),\\
\intertext{and if $\mathcal E_1,\mathcal E_2$ are supersingular,}
\toSepIsogssu((\mathcal E_1,\omega_1),(\mathcal E_2,\omega_2))
&=\IHom_{\D_{\Fpi}}(\TM_p\mathcal E_2,\TM_p\mathcal E_1).
\end{align*}

These morphism sets all have a natural profinite topologies,
and composition of morphisms is continuous. These categories
are `formal schemes' in the sense of Devinatz~\cite{Devinatz}
and Strickland~\cite{Strickland} and we will make use of this
in Section~\ref{sec:CohTopHA}. Their object sets have the form
$\Spec_{\F_p}\ssEll_*$, while the morphism set of a pair of
objects can be identified with the limit of the pro-system
obtained by factoring out by the open neighbourhoods of the
identity morphisms in the sets of homomorphisms between the
associated Tate modules. For example,
\[
\tIsog(\mathcal E_1,\mathcal E_2)
=
\invlim{U}\Hom_{\D_{\Fpi}}(\TM_p\mathcal E_2,\TM_p\mathcal E_1/U),
\]
where the limit is taken over the basic neighbourhoods $U$ of $0$
in $\TM_p\mathcal E_1$ which are just the finite index subgroups.
We can describe representing algebras for some of these formal
schemes. We will do this for the supersingular category $\toSepIsogssu$.
Recall Theorem~\ref{thm:GammaReps}.
\begin{thm}\label{thm:Rep-tIsogssu}
There is an equivariant isomorphism of topological groupoids
with $\Gm$-action,
\[
\Specf_{\Fpi}\Fpi\otimes\ssGamma_*\iso\toSepIsogssu.
\]
Moreover, $\Fpi\otimes\ssGamma_{2n}$ can be identified with the
set of all continuous functions $\toSepIsogssu\lra\Fpi$ of weight~$n$
and $\ssGamma_{2n}\subset\Fpi\otimes\ssGamma_{2n}$ can be identified
with the subset of Galois invariant functions.
\end{thm}
\begin{proof}
This follows from an argument similar to that of~\cite{padic}; see
also~\cite{left} for a similar generalization. The idea is to consider
locally constant functions
\[
\toSepIsogssu((\mathcal E_1,\omega_1),(\mathcal E_2,\omega_2))
\iso
\W(\F_{p^2})\<\rmS\>
\lra\Fpi,
\]
where $(\mathcal E_1,\omega_1)$, $(\mathcal E_2,\omega_2)$ are two
elliptic curves. The space of all such functions is determined using
the ideas of~\cite{padic}, and turns out to be spanned by monomials
in generalized Teichm\"uller functions relative to expansions in terms
of powers of $\rmS$. Up to powers of $u$, these Teichm\"uller functions
are the images under the natural map $\Ell_*\Ell\lra\Fpi\otimes\ssGamma_{2n}$
of the elements $D_r\in\Ell_*\Ell$ of~\cite[equation~(9.8)]{ellellp1}.
\end{proof}

For later use we provide a useful example of such a Galois invariant
continuous function on $\toSepIsogssu$, namely
\begin{equation}
\label{ex:ind}
\ind\:\toSepIsogssu\lra\Fpi;
\quad
\ind((\mathcal E,\omega)\xrightarrow{\phi}(\mathcal E',\omega'))
=\deg\phi\modp{p},
\end{equation}
where $\deg\phi\modp{p}$ is calculated by choosing a separable isogeny
$\phi_0\:(\mathcal E,\omega)\lra(\mathcal E',\omega')$ which approximates
$\phi$ in $\Hom_{\D_{\Fpi}}(\TM_p\mathcal E',\TM_p\mathcal E)$ in the
sense that $\phi\equiv\phi_0\modp{S}$. Clearly $\ind$ is locally constant,
hence continuous, as well as Galois invariant. Also, for a composable
pair of morphisms $\phi,\theta$,
\[
\ind(\phi\theta)=\ind(\phi)\ind(\theta).
\]
\begin{prop}\label{prop:ind}
The function $\ind$ corresponds to a element of $\ssGamma_0$ which
is group-like under the coaction $\psi$ and antipode $\chi$ in the
sense that
\[
\psi(\ind)=\ind\otimes\ind,
\quad
\chi(\ind)=\ind^{-1}.
\]
\end{prop}

\section{Continuous cohomology of profinite groupoids}
\label{sec:ContCohProGpds}

The results of this section appeared in greater detail in~\cite{CohGpds}.

Let $\Gpd$ be a groupoid, \ie, a small category in which every morphism
is invertible. The function $\Obj\Gpd\lra\Mor\Gpd$ which sends each
object to its identity morphism embeds $\Obj\Gpd$ into $\Mor\Gpd$ so we
can view $\Gpd$ as consisting of the set of all its morphisms. We will
use the notation
\[
\Gpd(*,x)=\bigcup_{y\in\Obj\Gpd}\Gpd(y,x)\subset\Gpd.
\]
Following Higgins~\cite{Higgins} we define a subgroupoid $\GpdN$ of a
groupoid $\Gpd$ to be \emph{normal} if it satisfies the following
conditions: \\
A) $\Obj\GpdN=\Obj\Gpd$; \\
B) for every morphism $x\xrightarrow{f}y$ in $\Gpd$,
\[
f\GpdN(x,x)f^{-1}=\GpdN(x,x).
\]
We will write $\GpdN\ideal\Gpd$ if $\GpdN$ is a normal subgroupoid of
$\Gpd$. When $\GpdN\ideal\Gpd$ we can form the \emph{quotient groupoid}
$\Gpd/\GpdN$ whose objects are equivalence classes of objects of $\Gpd$
under the relation $\sim$ for which
\[
x\sim y
\iff
\text{$\exists$ $x\xrightarrow{h}y$ in $\GpdN$},
\]
and whose morphisms are equivalence classes of morphisms under the relation
\[
x\xrightarrow{f}y\sim x'\xrightarrow{f'}y'
\iff
\exists\;x\xrightarrow{p}x',\,y\xrightarrow{q}y'
\text{ in $\GpdN$ s.t. }f'=qfp^{-1}.
\]
Whenever two classes contain composable elements, composition of equivalence
classes can be defined by
\[
[y\xrightarrow{g}z][x\xrightarrow{f}y]=[x\xrightarrow{gf}z],
\]
which is well defined since for morphisms $x\xrightarrow{p}x'$,
$y\xrightarrow{q}y'$, $y\xrightarrow{r}y'$, $z\xrightarrow{s}z'$,
\[
(sgr^{-1})(qfp^{-1})=sg(r^{-1}q)fp^{-1}=(sh)gfp^{-1}
\]
where $h=g(r^{-1}q)g^{-1}$ is in $\GpdN(z,z)$. There is a quotient functor
$\Gpd\lra\Gpd/\GpdN$.

We define a groupoid $\Gpd$ to be \emph{automorphism finite} if for every
$x\in\Obj\Gpd$, $\Gpd(x,x)$ is a finite group. We define a groupoid $\Gpd$
to be (\emph{automorphism}) \emph{profinite} if it is the inverse limit of
automorphism finite groupoids,
\[
\Gpd\iso\invlim{\GpdN\ideal\Gpd}\Gpd/\GpdN.
\]
Such a groupoid has a natural topology in which the basic open sets have
the form
\[
U(x\xrightarrow{f}y,\GpdN)
=\{x\xrightarrow{g}y:gf^{-1}\in\GpdN(y,y)\},
\]
where $f\in\Mor\Gpd$ and $\Gpd/\GpdN$ is automorphism finite.

A \emph{topological groupoid} is a groupoid which is a topological space
such that the partial composition $\Gpd\Times_{\Obj\Gpd}\Gpd\lra\Gpd$,
inverse function $\Gpd\lra\Gpd$, domain and codomain functions
$\Gpd\lra\Obj\Gpd$ are all continuous, where $\Obj\Gpd$ has the subspace
topology. A profinite groupoid in the above sense is a topological groupoid.

A groupoid $\Gpd$ is \emph{connected} if for every pair of objects $x,y$
in $\Gpd$ there is a morphism $x\xrightarrow{f}y$. More generally, $\Gpd$
is the disjoint union of its \emph{connected components} which are the
connected subgroupoids.

In order to define the cohomology of profinite groupoids we first need
to define suitable a notion of module, and we follow the ideas of Galois
cohomology, accessibly described in Shatz~\cite{Shatz}, with Weibel~\cite{Weibel}
providing a more general cohomological discussion.

For a profinite groupoid $\Gpd$, a \emph{proper $\Gpd$-module} (over
a commutative unital ring $\k$) is a functor $\ubar M\:\Gpd\lra\Mod_\k$
in which for $x,y\in\Obj\Gpd$, $m\in\ubar M(x)$, $f\in\Gpd(x,y)$, the
set
\[
\Stab_{\Gpd}(m,f,y)=\{g\in\Gpd(x,y):\ubar M(g)m=\ubar M(f)m\}
\subset\Gpd(x,y)
\]
is open. This generalizes the notion of proper module for a profinite
group, for which stabilizers of points are of finite index. We will denote
the category of all proper $\Gpd$-modules over $\k$ by $\Mod_{\k,\Gpd}$,
where the morphisms are natural transformations.

We set
\[
\ubar{\mathbf M}=\Prod_{x\in\Obj\Gpd}\ubar M(x)
\]
and view this as a discrete topological space. We define a \emph{section}
of $\ubar M$ to be a function $\Phi\:\Obj\Gpd\lra\ubar{\mathbf M}$ such
that
\[
\Phi(x)\in\ubar M(x)\quad(x\in\Obj\Gpd).
\]
We will denote the $\k$-module of all sections by $\Sect(\Gpd;\ubar M)$.
\begin{prop}\label{prop:GpdMod-AbCat}
Let $\Gpd$ be a profinite groupoid and $\k$ a commutative unital ring.
\begin{enumerate}
\item[a)]
$\Mod_{\k,\Gpd}$ is an abelian category with structure inherited from
that of\/ $\Mod_\k$.
\item[b)]
$\Mod_{\k,\Gpd}$ has sufficiently many injectives.
\end{enumerate}
\end{prop}
\begin{proof}
These are straightforward generalizations of the analogous results for
profinite groups found in~\cite{Shatz}; further details appear in~\cite{CohGpds}.
\end{proof}

We will consider two functors $\Mod_{\k,\Gpd}\lra\Mod_\k$. The first is
a sort of fixed point construction
\[
(\ )^\Gpd\:\ubar M\rightsquigarrow\ubar M^\Gpd
=
\{\Phi\in\Sect(\Gpd;\ubar M):
\forall f\in\Gpd,\;\ubar M(f)(\Phi(\dom f))=\Phi(\codom f)\}.
\]
This functor is left exact, so it has right derived functors
$\mathrm{R}^n(\ )^\Gpd$ for $n\geq0$ which can be viewed as the continuous
cohomology of $\Gpd$, $\Hc^n(\Gpd;\ )$.

For $x_0\in\Obj\Gpd$, we define a kind of `localization at $x_0$',
\[
(\ )_{x_0}\:\ubar M\rightsquigarrow\ubar M(x_0)^{\Gpd(x_0,x_0)},
\]
obtained by restricting to $\ubar M(x_0)$ and taking the fixed points
under the action of the automorphism group of $x_0$. This is left exact
and has right derived functors $\mathrm{R}^n(\ )_{x_0}$ which we will
denote by $\mathrm{H}_{x_0}^n(\Gpd;\ )$.
\begin{prop}\label{prop:DerFunct}
If the profinite groupoid $\Gpd$ is connected, then for any $x_0\in\Obj\Gpd$,
there is a natural equivalence of functors
\[
(\ )_{x_0}\iso(\ )^\Gpd.
\]
Hence there are natural equivalences of functors
\[
\Hc^n(\Gpd;\ )\iso\mathrm{H}_{x_0}^n(\Gpd;\ )\quad(n\geq0).
\]
\end{prop}
\begin{proof}
For the first part, we produce a $\k$-isomorphism
$F\:\ubar M_{x_0}\xrightarrow{\iso}\ubar M^\Gpd$. For $m\in\ubar M_{x_0}$,
define $F(m)=\Phi_m$ by
\[
\Phi_m(x)=\ubar M(f)m\quad(x\in\Obj\Gpd)
\]
where we choose \emph{any} $f\in\Gpd(x_0,x)$; this choice does not affect
the outcome since for a second choice $g\in\Gpd(x_0,x)$,
$f^{-1}g\in\Gpd(x_0,x_0)$ and therefore
\[
\ubar M(g)m=\ubar M(f)\ubar M(f^{-1}g)m=\ubar M(f)m.
\]
This is easily verified to be an isomorphism.

The second part follows using a standard result on $\delta$-functors
described in~\cite[chapter 2]{Weibel}.
\end{proof}
This result shows that the calculation of the continuous cohomology
$\Hc^*(\Gpd;\ubar M)$ of a proper module $\ubar M$ reduces to continuous
group cohomology.
\begin{thm}\label{thm:GpdCoh-ChRings-Conn}
If $\Gpd$ is a connected profinite groupoid and $\ubar M$ is a proper
$\Gpd$-module over $\k$, then for any $x_0\in\Obj\Gpd$,
\[
\Hc^*(\Gpd;\ubar M)\iso\Hc^*(\Gpd(x_0,x_0);\ubar M(x_0)).
\]
\end{thm}
In the general case we have the following.
\begin{thm}
\label{thm:GpdCoh-ChRings-Gen}
If $\Gpd$ is a profinite groupoid and $\ubar M$ is a proper $\Gpd$-module
over $\k$, then for each $n\geq0$,
\[
\Hc^n(\Gpd;\ubar M)
\iso
\prod_C\Hc^n(\Gpd(x_C,x_C);\ubar M(x_C)),
\]
where $C$ ranges over the connected components of $\Gpd$ and $x_C$ is a
chosen object of $C$.
\end{thm}
\begin{rem}
\label{rem:TopSp}
The familiar approach to proving results of this kind would require there
to be a topological splitting of $\Gpd$ as $\Obj\Gpd\sdp\Gpd(x_0,x_0)$.
Indeed, in our application no such continuous splitting exists and we would
need to pass to a quotient category to make use of such an argument.
\end{rem}
\begin{rem}
\label{rem:Lim}
The fixed point functor $(\ )^\Gpd$ agrees with the limit over the category
$\Gpd$, hence the continuous cohomology functors $\Hc^n(\Gpd;\ )$ are the
derived functors of the limit. This connects our work with recent results
of Hopkins, Mahowald \etal. on the spectra $\EO_2$ and $\TMF$, the ring of
topological modular forms.
\end{rem}

\section{On the cohomology of topological Hopf algebroids}\label{sec:CohTopHA}

The best reference for the following material is Devinatz~\cite{Devinatz}
which contains the most thorough discussion we are aware of on the continuous
cohomology of topological groupoid schemes of the type under consideration
in this paper. We simply sketch the required details from~\cite[\S1]{Devinatz},
amending them slightly to suit our needs. An alternative perspective on the
cohomology of categories is provided by Baues \& Wirsching~\cite{Baues-Wirsching}.

Let $\mathcal C$ be a groupoid, $\k$ a commutative unital ring. Then $\Modc_\k$
will denote the category of \emph{complete Hausdorff $\k$-modules} defined in
\cite[Definition 1.1]{Devinatz}. Similarly, $\Algc_\k$ will denote the category
of \emph{complete commutative $\k$-algebras}. We will refer to morphisms in these
two categories as \emph{continuous homomorphisms} of complete modules or algebras.

A \emph{cogroupoid object} in $\Algc_\k$ is then a pair $(A,\Gamma)$ of objects
in $\Algc_\k$ together with the usual structure maps $\eta_R,\eta_L,\epsilon,\psi,\chi$
for a Hopf algebroid over $\k$ except that in all relevant diagrams the completed
tensor product $\cTensor{\k}$ has to be used; such data
$(A,\Gamma,\eta_R,\eta_L,\epsilon,\psi,\chi)$ will be referred to as constituting
a \emph{complete topological Hopf algebroid}. Equivalently, $\mathcal C=\Specf_\k\Gamma$
is an (affine) \emph{formal groupoid scheme}.

We can also consider the category $\Comodc_\Gamma$ of complete (left)
$\Gamma$-comodules with morphisms being comodule morphisms also lying in $\Modc_\k$.
For two objects $M$, $N$ in $\Comodc_\Gamma$, we will denote the set of morphisms
$M\lra N$ by $\Hom_\Gamma(M,N)$. The functor
\[
\Comodc_\Gamma\lra\Mod_\k;
\quad
M\rightsquigarrow\Hom_\Gamma(A,M)
\]
is left exact and its right derived functors form a graded functor $\Ext^*_\Gamma(A,\ )$.

Given a continuous morphism of complete Hopf algebroids $f\:(A,\Gamma)\lra(B,\Sigma)$
and a $\Gamma$-comodule $M$, there is an induced map
\[
H^*f\:\Ext^*_\Gamma(A,M)\lra\Ext^*_\Sigma(B,f^*M)
\]
where $f^*M=B\cTensor{A}M$ has the $\Sigma$-comodule structure described
in~\cite{Devinatz}.

Now recall from~\cite[definition 1.14]{Devinatz} the notion of a \emph{natural
equivalence} $\tau\:f\lra g$ between two continuous morphisms
$f,g\:(A,\Gamma)\lra(B,\Sigma)$ of complete Hopf algebroids. In particular,
such a $\tau$ induces a continuous homomorphism of complete $\Sigma$-comodules
$\tau^*\:g^*M\lra f^*M$, which in turn induces a map
\[
H^*\tau\:\Ext_\Sigma^*(B,g^*M)\lra\Ext_\Sigma^*(B,f^*M).
\]
We will require Devinatz's important Proposition~1.16.
\begin{prop}
\label{prop:EqceGpdsExtIso}
Let $\tau\:f\lra g$ be a natural equivalence between continuous morphisms
$f,g\:(A,\Gamma)\lra(B,\Sigma)$ of complete Hopf algebroids. Then for any
continuous $\Gamma$-comodule $M$,
\[
H^*(\tau^*)\o H^*g=H^*f\:
\Ext_\Gamma^*(A,M)\lra\Ext_\Sigma^*(B,f^*M).
\]
\end{prop}
Using Devinatz's notion of \emph{equivalence} of two complete Hopf algebroids
we can deduce
\begin{cor}
\label{cor:EqceGpdsExtIso}
If $f\:(A,\Gamma)\lra(B,\Sigma)$ is an equivalence of complete Hopf algebroids,
then
\[
H^*f\:\Ext_\Gamma^*(A,M)\lra\Ext_\Sigma^*(B,f^*M)
\]
is an isomorphism.
\end{cor}

All of the above can also be reworked with graded $\k$-modules and $\k$-algebras.
As Devinatz observes, when $\Gamma$ is concentrated in even degrees, this is
equivalent to introducing actions of the multiplicative group scheme $\Gm$ which
factor through the quotient scheme $\Gm/\mu_2$ where $\mu_d\subset\Gm$ is subscheme
of $d$-th roots of unity. This applies to the situations of interest to us and
indeed locally the actions of $\Gm$ factor through $\Gm/\mu_{2n}$ where may take
some of the values $n=2,4,6$. An alternative approach to this is to view a
$\Z$-graded module as a $\Z/2$-graded view module with action of $\Gm$, and we
prefer this approach. We then define an element $x$ in degree~$n$ to be of
\emph{weight} $\wt x=n/2$; this means that whenever $\alpha\in\Gm$,
\[
\alpha\.x=
\begin{cases}
\ph{-}\alpha^{\wt x}x&\text{if $n$ even},\\
-\alpha^{\wt x-1}x&\text{if $n$ odd}.
\end{cases}
\]

For a profinite groupoid $\Gpd$ and a commutative unital ring $\k$ with
the discrete topology, the pair of $\k$-algebras
\[
A=\Mapc(\Obj\Gpd,\k),\quad\Gamma=\Mapc(\Gpd,\k)
\]
form a complete topological Hopf algebroid $(A,\Gamma)$ with obvious
structure maps induced from $\Gpd$. A proper $\Gpd$-module $\ubar M$
with each $\ubar M(x)$ finitely generated over $\k$ is equivalent to
a topological $\Gamma$-comodule $M$, where
\[
M=\Sect(\Gpd;\ubar M)
\]
is a left $A$-module via
\[
(\alpha\.\Phi)(x)=\alpha(x)\Phi(x),
\]
and the coproduct $\psi\:M\lra\Gamma\oTimes_{A}M$ is determined on an
element $m\in M$ by an expression of the form
\[
\psi(m)=\sum_s\theta_s\otimes m_s,
\]
where $m_s\in M$ and $\theta_s\in\Gamma$ so that for each pair of
elements $x,y\in\Obj\Gpd$, $\theta_s$ restricts to a locally constant
function on $\Gpd(x,y)$.

\section{Connectivity of the category of supersingular isogenies}
\label{sec:Connectivity}

In this section we will show that the category of isogenies of
supersingular elliptic curves $\Isogss$ is connected in the sense
that there is a morphism between any given pair of objects.
\begin{thm}
\label{thm:ZetaFns}
Two elliptic curves $\mathcal E$, $\mathcal E'$ defined over a finite
field $\F_{p^d}$ are isogenous over $\F_{p^d}$ if and only if
$|\mathcal E(\F_{p^d})|=|\mathcal E'(\F_{p^d})|$.
\end{thm}
\begin{proof}
See \cite[Chapter 3 Theorem 8.4]{Husemoller}.
\end{proof}
\begin{thm}
\label{thm:ZetaFnN1}
An elliptic curve $\mathcal E$ defined over a finite field $\F_{p^d}$
satisfies
\begin{itemize}
\item
$|\mathcal E(\F_{p^d})|=1+p^d$ if $d$ is odd;
\item
$|\mathcal E(\F_{p^{2m}})|=(1\pm p^m)^2$ if
$\End_{\Fpi}\mathcal E=\End_{\F_{p^d}}\mathcal E$.
\end{itemize}
\end{thm}
\begin{proof}
The list of all possible orders of $|\mathcal E(\F_{p^d})|$
appears in~\cite{Waterhouse}, while~\cite{Ruck} gives a
complete list of the actual groups $\mathcal E(\F_{p^d})$
that can occur.
\end{proof}
\begin{thm}
\label{thm:ConnIsogss}
The isogeny categories $\Isogss$ and $\tIsogss$ are connected.
\end{thm}
\begin{proof}
We will show that any two supersingular elliptic curves
$\mathcal E$, $\mathcal E'$ defined over $\Fpi$ are isogenous.
We may assume that $\mathcal E$ and $\mathcal E'$ are
both defined over some finite field and then by Theorems
\ref{thm:ZetaFns} and \ref{thm:ZetaFnN1} we need only
show that they become isogenous over some larger finite
field. We begin by enlarging the common field of definition
to $\F_{p^{2m}}$ for which
\[
\End_{\Fpi}\mathcal E=\End_{\F_{p^{2m}}}\mathcal E
,\quad
\End_{\Fpi}\mathcal E'=\End_{\F_{p^{2m}}}\mathcal E'.
\]
Thus $|\mathcal E(\F_{p^{2m}})|$ and $|\mathcal E'(\F_{p^{2m}})|$
both have the form $(1\pm p^m)^2$. If these are equal
then the curves are isogenous over $\F_{p^{2m}}$.
Otherwise we may assume that
\[
|\mathcal E(\F_{p^{2m}})|=1+2p^m+p^{2m}
,\quad
|\mathcal E'(\F_{p^{2m}})|=1-2p^m+p^{2m}.
\]
If the Weierstra\ss{} equation of $\mathcal E$ is
\[
\mathcal E\:y^2=4x^3-ax-b,
\]
taking a quadratic non-residue $u$ in $\F_{p^{2m}}$
allows us to define a twisted curve by
\[
\mathcal E^u\:y^2=4x^3-u^2ax-u^3b,
\]
which becomes isomorphic to $\mathcal E$ over
$\F_{p^{4m}}$. If
\begin{align*}
N_0&=|\{t\in\F_{p^{2m}}:4t^3-at-b=0\}|,\\
N_1&=
|\{t\in\F_{p^{2m}}:
\text{$4t^3-at-b\neq0$ is a quadratic residue}\}|,
\end{align*}
then
\[
1+N_0+2N_1=1+2p^m+p^{2m}.
\]
But as
\[
4x^3-u^2ax-u^3b=u^3(4(u^{-1}x)^3-a(u^{-1}x)-b),
\]
we find that
\begin{align*}
|\mathcal E^u(\F_{p^{2m}})|&=1+N_0+2(p^{2m}-N_0-N_1)\\
&=1-N_0-2N_1+2p^{2m}\\
&=1-2p^m+p^{2m}.
\end{align*}
Hence,
\[
|\mathcal E'(\F_{p^{2m}})|=|\mathcal E^u(\F_{p^{2m}})|
\]
and so these are isogenous curves over $\F_{p^{2m}}$,
implying that $\mathcal E'$ is isogenous to $\mathcal E$
over $\F_{p^{2m}}$.

We could have also used the fact $j(\mathcal E^u)=j(\mathcal E)$
to obtain an isomorphism $\mathcal E^u\iso\mathcal E$ over
$\Fpi$, but the argument given is more explicit about the
field of definition of such an isomorphism.

The connectivity of $\tIsogss$ now follows from Tate's
Theorem~\ref{thm:Tate}.
\end{proof}
\begin{cor}
\label{cor:ConnIsogss}
The groupoids $\Isogssu$ and $\tIsogssu$ are connected.
\end{cor}

The following deeper fact about supersingular curves over finite
fields, which is a consequence of Theorem~\ref{thm:ssFp}, allows
us to show the connectivity of $\tSepIsogssu$.
\begin{thm}
\label{thm:ssFp-existence}
For any prime $p>3$, there is a supersingular elliptic curve
$\mathcal E_0$ defined over $\F_p$. If $p>11$, this can be
chosen to satisfy $j(\mathcal E)\not\equiv0,1728\modp p$.
\end{thm}
\begin{prop}\label{prop:tSIsogssu}
The separable isogeny categories $\SepIsogssu$ and $\tSepIsogssu$
are connected as are the associated categories of isogenies
of oriented elliptic curves $\oSepIsogssu$ and $\toSepIsogssu$.
\end{prop}
\begin{proof}
Choose a supersingular curve $\mathcal E_0$ defined over $\F_p$
as in Theorem~\ref{thm:ssFp-existence}. For each supersingular
curve $\mathcal E$ defined over $\Fpi$ there is an isogeny
$\phi\:\mathcal E\lra\mathcal E_0$. By Proposition~\ref{prop:IsogFactors},
there is a factorization
\[
\phi=\Fr^k\o\phi_s,
\]
where $\phi_s\:\mathcal E\lra\mathcal E_0^{(1/p^k)}$ is separable.
But $\mathcal E_0^{(1/p^k)}=\mathcal E_0$ since $\mathcal E_0$ is
defined over $\F_p$, hence $\phi_s\:\mathcal E\lra\mathcal E_0$ is
a separable isogeny connecting $\mathcal E$ to $\mathcal E_0$. Thus
$\SepIsogssu$ is connected.

Now Tate's Theorem~\ref{thm:Tate} implies that $\tSepIsogssu$ is
connected.

The results for $\oSepIsogssu$ and $\toSepIsogssu$ follow by twisting.
\end{proof}

These results have immediate implications for the cohomology of the
groupoids $\Isogssu$ and $\SepIsogssu$, however, for our purposes with
$\toSepIsogssu$ we need to take the topological structure into account
and consider an appropriate continuous cohomology. We will discuss this
further in the following sections.

\section{Splittings of a quotient of the supersingular category of isogenies}
\label{sec:Splittings}

In this section we introduce some quotient categories of $\toSepIsogssu$.
The first is perhaps more `geometric', while the second is a `$p$-typical'
approximation.

Our first quotient category is $\catC=\toSepIsogssu/\AUT$, where $\AUT$
denotes the automorphism subgroupoid scheme of $\toSepIsogssu$ which is
defined by taking the collection of automorphism groups of all the objects
of $\toSepIsogssu$,
\[
\AUT=\coprod_{(\mathcal E,\omega)}\Aut\mathcal E.
\]
Notice that the automorphism group of $(\mathcal E,\omega)$ only depends
on $\mathcal E$ and so we can safely write $\Aut\mathcal E$ for this. The
objects of $\catC$ are the objects of $\tSepIsogssu$, whereas the morphism
sets are double cosets of the form
\[
\catC((\mathcal E_1,\omega_1),(\mathcal E_2,\omega_2))=
\Aut\mathcal E_2
\backslash\toSepIsogssu((\mathcal E_1,\omega_1),
(\mathcal E_2,\omega_2))/\Aut\mathcal E_1.
\]
If we denote the twisting automorphism corresponding to
$t\in\Aut\mathcal E\subset\Fpi^\times$ by
$\tau_t\:\mathcal E\lra\mathcal E^{t^2}$, where
\[
\tau_t(x,y)=(t^2x,t^3y),
\]
then an element of $\catC((\mathcal E_1,\omega_1),(\mathcal E_2,\omega_2))$
is an equivalence class of morphisms in $\toSepIsogssu$ of the form
\[
\tau_t\o\phi\o\tau_{s^{-1}}
\quad(s\in\Aut(\mathcal E_1),t\in\Aut(\mathcal E_2)),
\]
for some fixed morphism
$\phi\:(\mathcal E_1,\omega_1)\lra(\mathcal E_2,\omega_2)$.

Our second quotient category is $\catC_0=\tSepIsogssu/\mubf_{p^2-1}$,
where $\mubf_{p^2-1}$ denotes the \'etale subgroupoid scheme of
$\toSepIsogssu$ generated by all twistings by elements in the kernel
of the $(p^2-1)$-power map $\Gm\lra\Gm$, whose points over $\Fpi$ form
the group
\[
\mu_{p^2-1}(\Fpi)=\{t\in\Fpi^\times:t^{p^2-1}=1\}.
\]
Notice that $\AUT$ is a subgroupoid scheme of $\mubf_{p^2-1}$. Objects
of $\catC_0$ are equivalence classes $[\mathcal E,\omega]$ of objects of
$\toSepIsogssu$, and the morphism set
$\catC_0([\mathcal E_1,\omega_1],[\mathcal E_2,\omega_2])$ is a double
coset of the form
\[
\mu_{p^2-1}\backslash\toSepIsogssu((\mathcal E_1,\omega_1),
(\mathcal E_2,\omega_2))/\mu_{p^2-1}
\]
which is the equivalence class consisting of morphisms in $\toSepIsogssu$
of the form
\[
\tau_t(x,y)=(t^2x,t^3y),
\]
and an element of $\catC((\mathcal E_1,\omega_1),(\mathcal E_2,\omega_2))$
is an equivalence class of morphisms in $\toSepIsogssu$ of the form
\[
\tau_t\o\phi\o\tau_{s^{-1}}
\quad(s,t\in\mu_{p^2-1}),
\]
for some fixed morphism
$\phi\:(\mathcal E_1,\omega_1)\lra(\mathcal E_2,\omega_2)$.

The set of objects in $\catC_0$ is represented by the invariant
subring
\[
\ssEll_*[u,u^{-1}]^{\mu_{p^2-1}}\subset\ssEll_*[u,u^{-1}],
\]
where the action of $\mu_{p^2-1}$ is given by
\[
t\.xu^n=t^{d+n}x\quad(x\in\ssEll_{2d},\;t\in\mu_{p^2-1}(\Fpi)).
\]
Furthermore, the set of morphisms of $\catC_0$ is represented by
the algebra
\[
\ssGamma^{\mu_{p^2-1}}_*=
\ssEll_*[u,u^{-1}]^{\mu_{p^2-1}}\oTimes_{\epsilon}\ssGamma_*
\oTimes_{\epsilon}\ssEll_*[u,u^{-1}]^{\mu_{p^2-1}},
\]
where the tensor products are formed using the idempotent ring
homomorphism
\[
\epsilon\:\ssEll_*[u,u^{-1}]\lra\ssEll_*
\]
obtained by averaging over the action of $\mu_{p^2-1}$ whose
image is $\ssEll_*[u,u^{-1}]^{\mu_{p^2-1}}$.
\begin{thm}\label{thm:GammaMuReps}
There is a natural isomorphism of groupoids with $\Gm$-action,
\[
\Spec_{\Fpi}\Fpi\otimes\ssGamma^{\mu_{p^2-1}}_*\iso\catC_0.
\]
Moreover, $\Fpi\otimes\ssGamma^{\mu_{p^2-1}}_{2n}$ can be identified with
the set of continuous functions $\catC_0\lra\Fpi$ of weight~$n$ and
$\ssGamma^{\mu_{p^2-1}}_{2n}\subset\Fpi\otimes\ssGamma^{\mu_{p^2-1}}_{2n}$
with the subset of Galois invariant functions.

The natural morphism of topological groupoids
$\tilde\epsilon\:\toSepIsogssu\lra\catC_0$ is induced by the natural morphism
of Hopf algebroids
$\epsilon\:(\ssEll_*[u,u^{-1}]^{\mu_{p^2-1}},
\ssGamma^{\mu_{p^2-1}}_*)\lra(\ssEll_*,\ssGamma_*)$
under which $u$ goes to~$1$. Furthermore, $\tilde\epsilon$ is an equivalence
of topological groupoids.
\end{thm}
In the latter part of this result, the topological structure has to be taken
into account when discussing equivalences of groupoids, with all the relevant
maps required to be continuous. This fact will be used to prove some cohomological
results in Section~\ref{sec:EqcesHopfAlgbds}. Notice that $\mubf_{p^2-1}$ is
an \'etale group scheme and $\tilde\epsilon$ is an \'etale morphism.

By Proposition~\ref{prop:tSIsogssu}, $\toSepIsogssu$ is connected, hence so
are the quotient categories $\catC$ and $\catC_0$. The following stronger result
holds.
\begin{thm}\label{thm:TopConn-catC}
Let $\boldsymbol{\mathcal E}_0$ be an object of either of these categories.
Then there are continuous maps
$\sigma\:\catC\lra\Obj\catC$, $\sigma_0\:\catC_0\lra\Obj\catC_0$ for which
\[
\dom\sigma(\boldsymbol{\mathcal E})=\dom\sigma_0(\boldsymbol{\mathcal E})
=\boldsymbol{\mathcal E}_0,\quad
\codom\sigma(\boldsymbol{\mathcal E})=\codom\sigma_0(\boldsymbol{\mathcal E})
=\boldsymbol{\mathcal E}.
\]
Hence there are splittings of topological categories
\[
\catC
\iso\Obj\catC\rtimes\Aut_{\catC}\boldsymbol{\mathcal E}_0,
\quad
\catC_0
\iso\Obj\catC_0\rtimes\Aut_{\catC_0}\boldsymbol{\mathcal E}_0.
\]
\end{thm}
\begin{proof}
We verify this for $\catC$, the proof for $\catC_0$
being similar. Choose an object $(\mathcal E_0,\omega_0)$
of $\catC$ and set $\alpha_0=j(\mathcal E_0)$.

First note that for each $\alpha\in\Fpi$, the subcategory of
$\toSepIsogssu$ consisting of objects $(\mathcal E,\omega)$
with $j(\mathcal E)=\alpha$ is either empty or forms a closed
and open set $U_\alpha$ in the natural (Zariski) topology
on the space of all such elliptic curves. In each of the
non-empty sets $U_\alpha$, we may pick an element
$(\mathcal E_\alpha,\omega_\alpha)$. Then for each
$(\mathcal E,\omega)$ with $j(\mathcal E)=\alpha$, there is
a non-unique isomorphism
$\phi_{(\mathcal E,\omega)}\:
(\mathcal E_\alpha,\omega_\alpha)\lra(\mathcal E,\omega)$
in $\oSepIsog$. Given a second such isomorphism $\phi'$, the
composite $\phi^{-1}\o\phi'$ is in $\Aut\mathcal E_\alpha$.
Passing to the quotient category $\catC$ we see that the image
of the subcategory generated by $U_\alpha$ is connected since
all such isomorphisms $\phi$ have identical images.

Now for every $\alpha$ with $U_\alpha$ non-empty, we may
choose a separable isogeny
\[
\phi_\alpha\:(\mathcal E_0,\omega_0)
\lra(\mathcal E_\alpha,\omega_\alpha).
\]
Again, although this is not unique, on passing to the
image set $\bar U_\alpha$ in $\catC$ we obtain a unique
such morphism between the images in $\catC$. Forming the
composite $\phi_{(\mathcal E,\omega)}\o\phi_\alpha$ and
passing to $\catC$ gives a continuous map
$\bar U_\alpha\lra\catC$ with the desired properties,
and then patching together these maps over the finitely
many supersingular $j$-invariants for the prime $p$
establishes the result.
\end{proof}
\begin{cor}
\label{cor:Category-Splittings}
There are equivalences of topological categories
\[
\catC
\simeq\Aut_{\catC}\boldsymbol{\mathcal E}_0,
\quad
\catC_0
\simeq\Aut_{\catC_0}\boldsymbol{\mathcal E}_0.
\]
\end{cor}
\begin{cor}
\label{cor:HA-Splittings}
There is an equivalence of Hopf algebroids
\[
(\ssEll_*[u,u^{-1}]^{\mu_{p^2-1}},
\ssGamma^{\mu_{p^2-1}}_*)
\lra(K(2)_*,K(2)_*K(2)).
\]
\end{cor}

\section{Some equivalences of Hopf algebroids}
\label{sec:EqcesHopfAlgbds}

Now that we possess the machinery developed in
Section~\ref{sec:CohTopHA}, we are in a position to give some
cohomological results. Our goal is to reprove the following
result of~\cite[theorem 4.1]{ellext}.
\begin{thm}
\label{thm:ExtEllEll-ExtKnKn}
There is an equivalence of Hopf algebroids
\[
(\Ell_*/(p,A),\Ell_*\Ell/(p,A))\lra(K(2)_*,K(2)_*K(2)),
\]
inducing an isomorphism
\[
\Ext^{*,*}_{\Ell_*\Ell}(\Ell_*,\Ell_*/(p,A))
\iso
\Ext^{*,*}_{K(2)_*K(2)}(K(2)_*,K(2)_*).
\]
\end{thm}

First we will make use of a further result of Devinatz~\cite{Devinatz}.
\begin{prop}\label{prop:Dev-b0}
The natural morphism of Hopf algebroids
\[
(\ssEll_*[u,u^{-1}],\ssGamma_*)\lra(\ssEll_*,\ssGamma_*^0)
\]
induces an isomorphism of $\Ext$ groups.
\end{prop}
\begin{proof}
See the discussion of~\cite[Construction 2.7]{Devinatz}, in particular
the remarks between equations~(2.8) and~(2.9).
\end{proof}
By~\cite[proposition 1.3d]{MiRa}, we have
\begin{align*}
\Ext^{*,*}_{\Ell_*\Ell}(\Ell_*,\Ell_*/(p,A))
&\iso
\Ext^{*,*}_{\ssGamma_*^0}(\ssEll_*,\ssEll_*)\\
&\iso
\Ext^{*,*}_{\ssGamma_*}(\ssEll_*[u,u^{-1}],\ssEll_*[u,u^{-1}]).
\end{align*}
By Theorem \ref{thm:GammaMuReps}, there is an isomorphism
\begin{align*}
\Ext^{*,*}_{\ssGamma_*}(\ssEll_*[u,u^{-1}],\ssEll_*[u,u^{-1}])& \\
\iso&\Ext^{*,*}_{\ssGamma^{\mu_{p^2-1}}_*}
(\ssEll_*[u,u^{-1}]^{\mu_{p^2-1}},\ssEll_*[u,u^{-1}]^{\mu_{p^2-1}}).
\end{align*}
Finally, by Corollary~\ref{cor:HA-Splittings} there is an isomorphism
\begin{align*}
\Ext^{*,*}_{\ssGamma^{\mu_{p^2-1}}_*}
(\ssEll_*[u,u^{-1}]^{\mu_{p^2-1}},\ssEll_*[u,u^{-1}]^{\mu_{p^2-1}})
&\\
\iso&\Ext^{*,*}_{K(2)_*K(2)}(K(2)_*,K(2)_*).
\end{align*}

\section{Isogenies and stable operations in supersingular elliptic cohomology}
\label{sec:Isog&Ops}

In this section we explain how the category $\toSepIsogssu$ naturally provides a
model for a large part of the stable operation algebra of supersingular elliptic
cohomology $\ssEll^*(\ )$. In fact, it turns out that the subalgebra
$\ssEll^*\Ell=\ssEll^*(\Ell)$ can be described as a subalgebra of the `twisted
topologized category algebra' of $\toSepIsogssu$ with coefficients in
$\ssEll_*[u,u^{-1}]$. Such a clear description is not available for
$\Ell^*\Ell=\Ell^*(\Ell)$, although an analogous result for the stable operation
algebra $K(1)^*(E(1))$ is well known with the later being a twisted topological
group algebra. More generally, Morava and his interpreters have given analogous
descriptions of $K(n)^*(E(n))$ for $n\geq1$.

By \cite{homell}, $\ssEll_*$ and $\ssEll_*[u,u^{-1}]$
are products of `graded fields', hence
\begin{align*}
\ssEll^*\Ell&=\Hom_{\ssEll_*}(\ssEll_*(\Ell),\ssEll_*),\\
\ssEll[u,u^{-1}]^*\Ell&=
\Hom_{\ssEll_*[u,u^{-1}]}(\ssEll[u,u^{-1}]_*(\Ell),\ssEll_*[u,u^{-1}].
\end{align*}

The set of all morphisms originating at a particular
object $\UB{\mathcal E_0}=(\mathcal E_0,\omega_0)$ of
$\toSepIsogssu$, with defined over $\F_p$ (such curves
always exist by Theorem \ref{thm:ssFp}), is
\[
\toSepIsogssu(\UB{\mathcal E_0},*)=
\coprod_{\substack{\text{$\UB{\mathcal E}$ isogenous}\\
\text{to $\UB{\mathcal E_0}$}}}
\toSepIsogssu(\UB{\mathcal E_0},\UB{\mathcal E}),
\]
is noncanonically a product of the form
\[
\toSepIsogssu(\UB{\mathcal E_0},*)
=\coprod_{(\mathcal E_0/N,\omega_0)}
\IHom_{\D_\Fpi}(\TM_p(\mathcal E_0/N),\TM_p\mathcal E_0)
\times\Fpi^\times,
\]
where $N$ ranges over the finite subgroups of $\mathcal E_0$
of order prime to $p$.

Using similar methods to those of \cite{heckop,haell} we can
construct stable operations $\mathrm T_n$ ($p\nmid n$) in the
cohomology theory $\ssEll_*[u,u^{-1}]^*(\ )$ and these actually
restrict to operations in $\ssEll_*^*(\ )$. These operations
(together with Adams-like operations originating on the factor
of $\Fpi$) generate a Hecke-like algebra.

There is also a further set of operations coming from elements
of the component
\[
\toSepIsogssu(\UB{\mathcal E_0},\UB{\mathcal E_0})
=\IHom_{\D_\Fpi}(\TM_p(\mathcal E_0),\TM_p\mathcal E_0)
\times\Fpi^\times.
\]
These operations correspond to the Hecke-like algebra
of \cite{haell} associated to the Morava stabilizer group
$\SGp$.

Combining these two families of operations gives rise to
a composite Hecke-like algebra which in turn generates
subalgebras of the operation algebras in the cohomology
theories $\ssEll_*[u,u^{-1}]^*(\ )$ and $\ssEll_*^*(\ )$.

\section{Relationship with work of Robert}
\label{sec:Robert}

In~\cite{Robert}, Robert discussed the action of Hecke operators
$\mathrm T_n$ ($p\nmid n$) on the ring of holomorphic modular
forms modulo the supersingular ideal generated by $p,A$, in effect
studying the action of the \'etale part of the category
$\toSepIsogssu$. In this section we discuss connections between
this work and ours. Serre~\cite{Serre} gave an adelic description
of Hecke operators associated to supersingular elliptic curves which
appears to have connections with our work.

We begin with some comments on Robert's work which provides a
classical Hecke operator perspective on ours. We denote by
$\mathrm B\:\ssEll_*\lra\ssEll_*$ the operator $\mathrm B(F)=BF$,
which raises weight by $p+1$ and degree by $2(p+1)$. The following
operator commutativity formula holds for all primes $\ell\neq p$:
\begin{equation}
\label{eqn:B-Tl}
\mathrm B\mathrm T_\ell=\ell\mathrm T_\ell\mathrm B.
\end{equation}
This is actually a more general result than Robert proves
since he only works with holomorphic modular forms, but
our result of \cite{modforms}, discussed earlier in
Theorem~\ref{thm:pAv2}, gives in the ring $\ssEll_*$
\[
B^{p-1}=-\legendre{-1}{p}\Delta^{(p^2-1)/12},
\]
and this allows us to localize with respect to powers
of $\Delta$ or equivalently of $B$.

Recall the Hecke algebra
\[
\mathbf H_{p}=
\F_p[\mathrm T_\ell,\psi^\ell:\text{primes $\ell\neq p$}]
/\text{(relations)},
\]
where the relations are the usual ones satisfied
by Hecke operations, as described in Theorem~7
of \cite{heckop}. We follow Robert in introducing
certain twistings of a module $M$. For each natural
number $a$ let $M[a]$ denote underlying $\F_p$-module
$M$ with the twisted Hecke action
\[
\mathrm T_\ell\.m=\mathrm T_\ell^{[a]}m
=\ell^a\mathrm T_\ell m.
\]
When $F\in\ssEll_*$, this agrees with Robert's action
\[
\mathrm T_\ell^{[a]}F=\ell^a\mathrm T_\ell F,
\]
at least when restricted to the holomorphic part, and
then $M[a]\iso M[a+p-1]$ as $\mathbf H_{p}$-modules
since $\ell^{p-1}\equiv1\modp{p}$. We view multiplication
by $B$ as giving rise to homomorphisms of graded
$\mathbf H_{p}$-modules
$\mathrm B\:\ssEll_*[a]\lra\ssEll_*[a-1]$ for $a\in\Z$,
uniformly raising degrees by $2(p+1)$.

More generally, if $M_*$ is a right comodule over the
Hopf algebroid $(\ssEll_*,\ssGamma^0_*)$ with coproduct
$\gamma\:M\lra M\oTimes_{\ssEll_*}\ssGamma^0_*$, then
associated to each $a\in\Z$ there is a twisted comodule
$M_*[a]$ with coproduct
\[
\gamma^{[a]}m=\sum_im_i\otimes t_i\ind^a,
\]
where $\ind$ is defined in Equation~\ref{ex:ind} (see also
Proposition~\ref{prop:ind}) and $\gamma m=\sum_im_i\otimes t_i$.

Recall that for any left $\ssEll_*$-linear map
$\theta\:\ssGamma^0_*\lra\ssEll_*$ there is an operation
$\bar\theta$ on $M_*$ given by
\[
\bar\theta m=\sum_im_i\otimes\theta(t_i).
\]
By \cite{heckop,haell}, this construction gives rise to an
induced $\mathbf H_{p}$-module structures on $M_*$ and $M_*[a]$
agreeing with that generalizing Robert's discussed above. In
fact, these extend to module structures over the associated
\emph{twisted Hecke algebra} containing $\ssEll_*$ and
$\mathbf H_{p}$ as discussed in \cite{haell}. Also there are
homomorphisms $\mathrm B\:M_*[a]\lra M_*[a-1]$ of modules
over the twisted Hecke algebra and induced from multiplication
by $B$. The following is our analogue of~\cite[lemme~6]{Robert}.
\begin{thm}\label{thm:Robert-Comodule}
For $a\in\Z$, $\mathrm B\:M_*[a]\lra M_*[a-1]$
defines an isomorphism of $(\ssEll_*,\ssGamma^0_*)$-comodules.
\end{thm}
\begin{proof}
We make use of the description of $\Ell_*\Ell_{(p)}$ from
\cite{ellellp1} and view modular forms as functions on
the space of oriented lattices in $\C$. As will see, the
argument used by Robert to determine $\mathrm T_\ell B$
for a prime $\ell\neq p$ is based on a congruence in
$\Ell_*\Ell_{(p)}$.

By \cite[equation 19, th\'eor\`eme B/lemme 7]{Robert}, we
have the following. If for some $r=0,1,\ldots\ell-1$,
\[
L=\<\tau,1\>\subset\<\tau',1\>,
\quad
\tau'=\frac{\tau+r}{\ell},
\quad
q'=e^{\ds2\pi i\tau'},
\]
then taking $q$-expansions over $\Z_{(p)}$ gives
\[
B(q')-B(q)\equiv
-12\mathop{\sum\sum}_{\substack{1\leq s\leq\ell-1\\ 1\leq k}}
\(\frac{q^k(q')^s}{(1-q^k(q')^s)^2}+
\frac{q^k(q')^{-s}}{(1-q^k(q')^{-s})^2}\)
\modp{p}.
\]
On the other hand, if
\[
L=\<\tau,1\>\subset\<\tau,\frac{1}{\ell}\>,
\quad
\tau'=\ell\tau,
\quad
q'=e^{\ds2\pi i\tau'},
\]
then by taking $q$-expansions we obtain
\begin{multline*}
\ell^{p+1}B(q')-B(q)\equiv\\
\sum_{1\leq s\leq\ell-1}
\(
(\ell^2-1)B(q)
-12\ell^2\sum_{1\leq k}\(\frac{q^k(q')^s}{(1-q^k(q')^s)^2}+
\frac{q^k(q')^{-s}}{(1-q^k(q')^{-s})^2}\)
\)\modp{p}.
\end{multline*}
In the terminology of \cite{ellellp1}, the coefficient of each monomial
$q^u(q')^v$ is a stably numerical polynomial in $\ell$. Indeed, using
the integrality criterion of \cite[theorem 6.3]{ellellp1} for generalized
modular forms to lie in $\Ell_*\Ell_{(p)}$, together with the fact that
every lattice inclusion of index not divisible by $p$ factors into a
sequence of lattice inclusions of prime index, we can obtain similar
formul\ae{} for all lattice inclusions of (not necessarily prime) index
not divisible by $p$.

The precise interpretation of what is going on here is that there
are functions $F_0,F_1\in\Ell_*\Ell_{(p)}$ on inclusions of lattices
such that for any inclusion of lattices $L\subset L'$ of degree
$[L';L]$ not divisible by $p$,
\begin{equation}\label{eqn:Robert-ellellp1version}
B(L')-[L';L]B(L)=pF_0(L\subset L')+A(L)F_1(L\subset L').
\end{equation}
It follows from this that in the ring $\Ell_*\Ell_{(p)}$,
\begin{equation}\label{eqn:Robert-ellellp1version-coproduct}
\eta_R(B)-\eta_L(B)\ind\equiv0\modp{p,A_1},
\end{equation}
Notice that under the reduction map $\Ell_*\Ell_{(p)}\lra\ssGamma^0_*$,
the index function
\[
(L\subset L')\longmapsto[L';L]
\]
is sent to $\ind$. This can be seen as follows. For any supersingular
elliptic curve $\mathcal E$ defined over $\F_{p^2}$ there is an imaginary
quadratic number field $K$ in which $p$ is unramified and so there is
a lift $\alpha$ of $j(\mathcal E)$ contained in the ring of integers
$\mathcal O_K$. Then there is an elliptic curve $\tilde{\mathcal E}$
defined over $\mathcal O_K$ with $j(\tilde{\mathcal E})=\alpha$ and
reduction modulo~$p$ induces an isomorphism
$\tilde{\mathcal E}[n]\lra\mathcal E[n]$ for $p\nmid n$. Since a strict
(hence separable) isogeny $\phi\:\mathcal E\lra\mathcal E'$ of degree $n$,
and defined over an extension of $\F_{p^2}$ is determined by
$\ker\phi\subset\mathcal E[n]$, it can be lifted to a strict isogeny
$\tilde\phi\:\tilde{\mathcal E}\lra\tilde{\mathcal E'}$ of degree $n$ and
defined over some extension of $\mathcal O_K$, where $\ker\tilde\phi$ is
the preimage of $\ker\phi$ under reduction. Then $\ind\phi\equiv n\modp{p}$.
Hence if we express $\tilde{\mathcal E}$ in the form $\C/L$, $\tilde{\mathcal E'}$
can be realized as $\C/L'$ where $L\subset L'$ has index $n$.

Because $\ind^{p-1}=c_1$ (the constant function taking value 1), we have
\begin{equation}\label{eqn:Robert-ellellp1version-coproduct-p-1}
\eta_R(B^{p-1})-\eta_L(B^{p-1})\equiv0\modp{p,A_1},
\end{equation}
which implies that $B^{p-1}\in\ssEll_*$ is coaction primitive.

By Equation~\eqref{eqn:Robert-ellellp1version}, we have
\begin{align*}
\gamma\mathrm B(m)=\gamma(mB)&=\sum_im_i\otimes t_i\eta_RB\\
&=\sum_im_i\otimes Bt_i\ind\\
&=\sum_im_iB\otimes t_i\ind\\
&=\sum_i\mathrm B(m_i)\otimes t_i\ind\\
&=\mathrm B\gamma^{[1]}m,
\end{align*}
where we have viewed $M_*$ as a right $\ssEll_*$-module and used
$\ssEll_*$-bimodule tensor products.

The determination of $\mathrm T_\ell B$ now follows from our
definition of the Hecke operators of \cite[equation 6.5]{ellellp1},
as does the generalization of Robert's formula
\[
\mathrm T_\ell(BF)\equiv\ell B\mathrm T_\ell(F),
\]
which holds for $F\in\ssEll_*$, and primes $\ell\neq p$.
\end{proof}

Our results are more general than those of Robert since they involve
generalized isogenies rather than just isogenies to define
$\ssEll_*$-linear maps $\ssGamma^0_*\lra\Fpi$ and hence operations
$\bar\phi$ on $\ssGamma^0_*$-comodules. Explicit operations of this
type were defined in \cite{haell} using Hecke operators derived from
the space of double cosets
\[
\<\mu_{p^2-1},p\>\backslash\tSGp/\<\mu_{p^2-1},p\>
\]
and its associated Hecke algebra; in fact this space is homeomorphic
to $\SGp^0\rtimes\Z/2$ as a space. For each supersingular elliptic
curve $(\mathcal E,\omega)$ we can identify $\W(\F_{p^2})\<\rmS\>$
with $\tIsog((\mathcal E,\omega),(\mathcal E,\omega))$ and as in \cite{haell}
obtain for each $\alpha\in\SGp\rtimes\Z/2$ an $\ssEll_*$-linear map
$\alpha_*\:\ssGamma^0_*\lra\ssEll_*$ and hence an operation $\bar\alpha$
on $\ssGamma^0_*$-comodules. This can be further generalized by associating
to each positive integer $d$ and each separable isogeny
$(\mathcal E,\omega)\xrightarrow{\phi}(\mathcal E',\omega')$ of degree $d$,
the corresponding element
$\alpha^\phi\in\tIsog((\mathcal E,\omega),(\mathcal E',\omega'))$ and
then symmetrizing over all of these to form a $\ssEll_*$-linear map
\[
\alpha^d_*\:\ssGamma^0_*\lra\ssEll_*;
\quad
(\alpha^d_*F)(\mathcal E,\omega)=
\frac{1}{d}\sum_{\phi}\alpha^\phi_*.
\]

Robert analyzes the holomorphic part of $\ssEll_{2n}$ as
a $\mathbf H_{p}$-module, in particular he determines when
the \emph{Eisenstein modules} $\Eis_k$ embed, where $\Eis_k$
is the 1-dimensional $\F_p$-module on the generator $e_k$
for which
\[
\mathrm T_\ell e_k=(1+\ell^{k-1})e_k.
\]
Thus $\Eis_k$ is an eigenspace for each Hecke operator
$\mathrm T_\ell$, and there is an isomorphism of
$\mathbf H_{p}$-modules
\[
\Eis_{2k}\iso\F_p\{\tilde E_{2k}\}\subset\ssEll_{4k};
\quad
e_{2k}\longmapsto\tilde E_{2k},
\]
where $\tilde E_{2k}$ is the reduction of one of the
following elements of $(\Ell_{2k})_{(p)}$:
\[
\begin{cases}
\ph{(B_{2k}/4k)}E_{2k}&\text{if $(p-1)\mid2k$},\\
(B_{2k}/4k)E_{2k}&\text{if $(p-1)\nmid2k$}.
\end{cases}
\]
In particular, $\Eis_0$ is the `trivial' module for
which
\[
\mathrm T_\ell e_0=(1+\ell^{-1})e_0.
\]
Robert gives conditions on when there is an occurrence
of $\Eis_k$ in $\ssEll_{2n}$, at least in the holomorphic
part. Since localization with respect to powers of $\Delta$
is equivalent to that with respect to powers of $B$ by
the main result of \cite{modforms} we can equally well
apply his results to $\ssEll_{2n}$, obtaining the
following version of \cite[th\'eor\`eme 3]{Robert}.
\begin{thm}
\label{thm:Robert-Thm3}
For a prime $p\geq5$ and an even integer $k$, there is an
embedding of $\mathbf H_{p}$-modules $\Eis_k\lra\ssEll_{2n}$
if and only if one of the following congruences holds:
\[
n\equiv k\modp{p^2-1},\quad n\equiv pk\modp{p^2-1}.
\]
\end{thm}
Notice that in particular, the trivial module $\Eis_0$ occurs
precisely in degrees $2n$ for which $(p^2-1)\mid2n$.
The $\Ext$ groups of $\ssEll_*$ over $\ssGamma^0_*$ were
investigated in~\cite{ellext,ellcalc}, and the results show
that Robert's conditions are weaker than needed to calculate
$\Ext^0$. Of course, his work ignores the effect of operations
coming from the `connected' part of $\toSepIsogssu$.

\appendix
\section{Supersingular curves defined over $\F_p$}
\label{sec:SSFp}

For every prime $p>3$ with $p\not\equiv1\modp{12}$ there are
supersingular elliptic curves defined over $\F_p$ since the
Hasse invariant then has $Q$ or $R$ as a factor. The following
stronger result which seems to be due to Deuring is also true
and a sketch proof can be found in~\cite{modforms}; Cox~\cite{Cox}
also contains an accessible account of related material.

Let $\mathcal E$ be a supersingular elliptic curve over $\Fpi$
whose $j$-invariant is $j(\mathcal E)\equiv0\modp{p}$ or
$1728\modp{p}$. Recall that the endomorphism ring $\End\mathcal E$
contains an imaginary quadratic number ring of the form
\[
\begin{cases}
\Z[\omega]&\text{if $j(\mathcal E)\equiv0\modp{p}$},\\
\Z[i]&\text{if $j(\mathcal E)\equiv1728\modp{p}$}.
\end{cases}
\]
By Theorem~\ref{thm:Standard WeierstarssForm}, such elliptic curves
are isomorphic to Weierstra\ss{} curves defined over~$\F_p$.

Let $K=\Q(\sqrt{-p})$ and $\mathcal O_K$ be its ring of integers
which is its unique maximal order.
\begin{thm}
\label{thm:ssFp}
For any prime $p>11$, there are supersingular elliptic curves
$\mathcal E$ defined over $\F_p$ with
$j(\mathcal E)\not\equiv0,1728\modp{p}$ and
$\mathcal O_K\subset\End\mathcal E$.
\end{thm}

\end{document}